\documentclass{amsart}
\usepackage{amssymb}
\usepackage[mathcal]{euscript}
\usepackage[cmtip,all]{xy}

\newcommand{\bydef}{:=}
\newcommand{\defby}{=:}

\newcommand{\vphi}{\varphi}
\newcommand{\veps}{\varepsilon}
\newcommand{\ul}[1]{\underline{#1}}

\newcommand{\wh}[1]{\widehat{#1}}
\newcommand{\wt}[1]{\widetilde{#1}}
\newcommand{\wb}[1]{\overline{#1}}

\newcommand{\sym}{\mathcal{H}}
\newcommand{\id}{\mathrm{id}}

\DeclareMathOperator*{\ot}{\otimes}

\DeclareMathOperator{\rank}{\mathrm{rank}} 

\newcommand{\cA}{\mathcal{A}}

\newcommand{\cC}{\mathcal{C}}
\newcommand{\cD}{\mathcal{D}}
\newcommand{\cG}{\mathcal{G}}

\newcommand{\cJ}{\mathcal{J}}
\newcommand{\cK}{\mathcal{K}}
\newcommand{\cL}{\mathcal{L}}

\newcommand{\cO}{\mathcal{O}}

\newcommand{\cR}{\mathcal{R}}
\newcommand{\cS}{\mathcal{S}}

\newcommand{\cU}{\mathcal{U}}
\newcommand{\cV}{\mathcal{V}}

\newcommand{\ZZ}{\mathbb{Z}}

\newcommand{\FF}{\mathbb{F}}

\newcommand{\chr}[1]{\mathrm{char}\,#1}

\DeclareMathOperator{\Hom}{\mathrm{Hom}}
\DeclareMathOperator{\End}{\mathrm{End}}
\DeclareMathOperator{\Alg}{\mathrm{Alg}}

\DeclareMathOperator{\Aut}{\mathrm{Aut}}
\DeclareMathOperator{\inaut}{\mathrm{Int}}
\DeclareMathOperator{\AAut}{\mathbf{Aut}}
\DeclareMathOperator{\Der}{\mathrm{Der}}
\DeclareMathOperator{\supp}{\mathrm{Supp}\,}

\newcommand{\Lie}{\mathrm{Lie}}

\newcommand{\Ad}{\mathrm{Ad}}

\newcommand{\So}{\mathfrak{so}}

\newcommand{\frso}{{\mathfrak{so}}}

\newcommand{\frgl}{{\mathfrak{gl}}}

\newcommand{\tri}{\mathfrak{tri}}

\newcommand{\GL}{\mathrm{GL}}

\newcommand{\Ort}{\mathrm{O}}

\newcommand{\PGO}{\mathrm{PGO}}
\newcommand{\Spin}{\mathrm{Spin}}
\newcommand{\TRI}{\mathrm{Tri}}

\newcommand{\Gs}{\mathbf{G}}

\newcommand{\GLs}{\mathbf{GL}}

\newcommand{\PGOs}{\mathbf{PGO}}
\newcommand{\Spins}{\mathbf{Spin}}
\newcommand{\TRIs}{\mathbf{Tri}}

\newcommand{\bmu}{\boldsymbol{\mu}}
\newcommand{\Qs}{\mathbf{Q}}

\newcommand{\alb}{\mathbb{A}} 

\newcommand{\LL}{\mathbb{L}}
\newcommand{\SIM}{\mathrm{Sim}}
\newcommand{\SIMs}{\mathbf{Sim}}
\newcommand{\Cl}{\mathfrak{Cl}}
\newcommand{\buno}{\mathbf{1}}
\newcommand{\GIII}{\Gamma^\mathrm{(III)}}
\newcommand{\Br}{\mathrm{Br}}

\newtheorem{theorem}{Theorem}
\newtheorem{proposition}[theorem]{Proposition}
\newtheorem{lemma}[theorem]{Lemma}
\newtheorem{corollary}[theorem]{Corollary}

\theoremstyle{definition}
\newtheorem{df}[theorem]{Definition}

\theoremstyle{remark}
\newtheorem{remark}[theorem]{Remark}

\begin{document}

\title{Gradings on the Lie algebra $D_4$ revisited}

\author[A. Elduque]{Alberto Elduque${}^\star$}
\address{Departamento de Matem\'{a}ticas
 e Instituto Universitario de Matem\'aticas y Aplicaciones,
 Universidad de Zaragoza, 50009 Zaragoza, Spain}
\email{elduque@unizar.es}
\thanks{${}^\star$supported by the Spanish Ministerio de Econom\'{\i}a y Competitividad---Fondo Europeo de Desarrollo Regional (FEDER) MTM2010-18370-C04-02 and MTM2013-45588-C3-2-P, and by the Diputaci\'on General de Arag\'on---Fondo Social Europeo (Grupo de Investigaci\'on de \'Algebra)}

\author[M. Kochetov]{Mikhail Kochetov${}^\dagger$}
\address{Department of Mathematics and Statistics,
 Memorial University of Newfoundland,
 St. John's, NL, A1C5S7, Canada}
\email{mikhail@mun.ca}
\thanks{${}^\dagger$supported by Discovery Grant 341792-2013 of the Natural Sciences and Engineering Research Council (NSERC) of Canada and a grant for visiting scientists by Instituto Universitario de Matem\'aticas y Aplicaciones, University of Zaragoza}



\date{}

\begin{abstract}
We classify group gradings on the simple Lie algebra $\cL$ of type $D_4$ over an algebraically closed field of characteristic different from $2$: 
fine gradings up to equivalence and $G$-gradings, with a fixed group $G$, up to isomorphism. For each $G$-grading on $\cL$, we also study graded $\cL$-modules 
(assuming characteristic $0$).
\end{abstract}

\maketitle

\section{Introduction}

In the past two decades, there has been much interest in gradings on simple Lie algebras by arbitrary groups --- see our recent monograph \cite{EKmon} and references therein. In particular, the classification of fine gradings (up to equivalence) on all finite-dimensional simple Lie algebras over an algebraically closed field of characteristic $0$ is essentially complete (\cite[Chapters 3--6]{EKmon}, \cite{E14}, \cite{YuExc}). For a given group $G$, the classification of $G$-gradings (up to isomorphism) on classical simple Lie algebras over an algebraically closed field of characteristic different from $2$ was done in \cite{BK10} (see also \cite[Chapter 3]{EKmon}), excluding type $D_4$, which exhibits exceptional behavior due to the phenomenon of triality. Although the case of $D_4$ is included in \cite{E09d} (see also \cite{DMV} and \cite[\S 6.1]{EKmon}), only fine gradings are treated there and the characteristic is assumed to be $0$. Since we do not see how to extend those arguments to positive characteristic, here we use an approach based on affine group schemes, which was also employed in \cite{BK10}. 

Let $\FF$ be the ground field. Except in the Preliminaries, we will assume $\FF$ {\em algebraically closed} and $\chr{\FF}\ne 2$. All vector spaces, algebras, tensor products, group schemes, etc. will be assumed over $\FF$ unless indicated otherwise. The superscript $\times$ will indicate the multiplicative group of invertible elements.

Recall that affine group schemes are representable functors from the category $\Alg_\FF$ of unital associative commutative algebras over $\FF$ to the category of groups --- we refer the reader to \cite{Wh}, \cite[Chapter VI]{KMRT} or \cite[Appendix A]{EKmon} for the background. Every (na\"\i ve) algebraic group gives rise to an affine group scheme. These are precisely the {\em smooth algebraic} group schemes, i.e., those whose representing (Hopf) algebra is finitely generated and reduced. In characteristic $0$, all group schemes are reduced, but it is not so in positive characteristic. We will follow the common convention of denoting the (smooth) group schemes corresponding to classical groups by the same letters, but using bold font to distinguish the scheme from the group (which is identified with the $\FF$-points of the scheme). It is important to note that this convention should be used with care: for example, the automorphism group scheme $\AAut_\FF(\cU)$ of a finite-dimensional algebra $\cU$ is defined by $\AAut_\FF(\cU)(\cS)=\Aut_\cS(\cU\ot\cS)$ for every $\cS$ in $\Alg_\FF$, and may be strictly larger than the smooth group scheme corresponding to the algebraic group $\Aut_\FF(\cU)$.

Let $\cL$ be a Lie algebra of type $D_4$. It is well known that the automorphism group scheme $\AAut_\FF(\cL)$ is smooth and we have a short exact sequence
\begin{equation}\label{eq:exact_D4}
\xymatrix{
\mathbf{1}\ar[r] & \PGOs^+_8\ar[r] & \AAut_\FF(\cL)\ar[r]^-{\pi} & \mathbf{S}_3\ar[r] & \mathbf{1}
}
\end{equation}
where $\PGOs^+_8$ is the group scheme of inner automorphisms (which corresponds to the algebraic group $\PGO^+_8$, see e.g. \cite[\S 12.A]{KMRT}) and $\mathbf{S}_3$ is the constant group scheme corresponding to the symmetric group $S_3$. This sequence can be split by identifying $\cL$ with the {\em triality Lie algebra} (see Definition \ref{df:tri} below) of the {\em para-Cayley algebra} $\cC$, i.e., the Cayley algebra equipped with the new product $x\bullet y=\bar{x}\bar{y}$, where juxtaposition denotes the usual product of $\cC$ and bar denotes its standard involution (see e.g. \cite[\S 4.1]{EKmon}). If we use the standard model for Lie algebras of series $D$, i.e., identify $\cL$ with the skew elements in $\cR=M_8(\FF)$ with respect to an orthogonal involution $\sigma$, then the restriction $\AAut_\FF(\cR,\sigma)\to\AAut_\FF(\cL)$ is a closed imbedding whose image coincides with the semidirect product of $\PGOs^+_8$ and the constant group scheme corresponding to one of the subgroups of order $2$ in $S_3$ (see e.g. \cite[\S 3.1]{EKmon}).

Now suppose that we have a grading $\Gamma:\;\cL=\bigoplus_{g\in G}\cL_g$ by a group $G$. Since the elements of the support of a group grading on a simple Lie algebra necessarily commute, we will always assume that $G$ is {\em abelian}. (Since the support is finite, we may also assume $G$ finitely generated.) Then $\Gamma$ is equivalent to a morphism $\eta=\eta_\Gamma\colon G^D\to\AAut_\FF(\cL)$ where $G^D$ is the Cartier dual of $G$ (i.e., the affine group scheme represented by the group algebra $\FF G$, so $G^D(\cS)=\Hom(G,\cS^\times)$ for every $\cS$ in $\Alg_\FF$) and $\eta$ is defined in Equation \eqref{eq:def_eta} below --- the details can be found in e.g. \cite[\S 1.4]{EKmon}. The image $\pi\eta(G^D)$ is an abelian subgroupscheme of $\mathbf{S_3}$. Since the subgroupschemes of a constant group scheme correspond to subgroups, here we have three possibilities: the image has order $1$, $2$ or $3$. The grading $\Gamma$ will be said to have {\em Type I, II or III}, respectively. (We use a capital letter here to distinguish from the other meaning of {\em type} common in the literature on gradings, which refers to the sequence of integers that records the number of homogeneous components of each dimension.) The subgroupscheme $\eta^{-1}(\eta(G^D)\cap\PGOs^+_8)$ of $G^D$ corresponds to a subgroup $H$ of $G$ of order $1$, $2$ or $3$, respectively. We will refer to $H$ as the {\em distinguished subgroup} of $\Gamma$. It is the smallest subgroup of $G$ such that the induced $G/H$-grading is of Type I.

If $\Gamma$ is of Type I then the image $\eta(G^D)$ lies in $\PGOs^+_8$, which is a subgroupscheme of index $2$ in $\PGOs_8=\AAut_\FF(\cR,\sigma)$. In the case of Type II, applying an outer automorphism of order $3$ if necessary, we may assume that $\eta(G^D)$ lies in $\AAut_\FF(\cR,\sigma)$. Therefore, any $G$-grading $\Gamma$ on $\cL$ of Type I or II is isomorphic to the restriction of a $G$-grading $\Gamma'$ on the algebra with involution $(\cR,\sigma)$. For this reason, we will sometimes collectively refer to gradings of Type I and II as {\em matrix gradings} (respectively, ``inner'' and ``outer''). They were classified in \cite{BK10} (see also \cite[\S 2.4]{EKmon}) up to {\em matrix isomorphism}, i.e., up to the action of $\Aut_\FF(\cR,\sigma)$. Note that there is a subtlety here: one isomorphism class of Type I gradings on $\cL$ may correspond to $1$, $2$ or $3$ isomorphism classes of gradings on $(\cR,\sigma)$ --- see Section \ref{s:Type_I}. No such difficulty arises for Type II gradings.

Our main concern in this paper are Type III gradings on $\cL$. We note that, since $\pi\eta(G^D)$ is a diagonalizable subgroupscheme, the group algebra of the corresponding subgroup of $S_3$ must be semisimple (being isomorphic to the Hopf dual of the group algebra of a finite subgroup of $G$). If $\chr\FF=3$ then the group algebra of the cyclic group of order $3$ is not semisimple (in other words, the corresponding constant group scheme is not diagonalizable), hence {\em Type III gradings  do not occur in characteristic~$3$}. 

An essential ingredient in our approach is the concept of {\em trialitarian algebra} introduced in \cite[Chapter X]{KMRT}, which is a central simple associative algebra over a cubic \'{e}tale algebra $\LL$ equipped with an orthogonal involution and some additional structure. The general definition is quite involved; here we will only need a special case of trialitarian algebras arising as endomorphisms of so-called {\em cyclic composition algebras} (see Definitions \ref{df:cyclic_comp} and \ref{df:tri_alg}). As shown in \cite[\S 45.C]{KMRT}, over any field of characteristic different from $2$, any simple Lie algebra $\cL$ of type $D_4$ is isomorphic to a canonically defined Lie subalgebra of a unique (up to isomorphism) trialitarian algebra $E$. 
Moreover, the restriction $\AAut_\FF(E)\to\AAut_\FF(\cL)$ is an isomorphism. This means, in particular, that any $G$-grading on $\cL$ is the restriction of a unique $G$-grading on $E$, hence $\cL$ and $E$ have the same classifications of gradings.

The structure of the paper is the following. In Section~\ref{s:preliminaries}, we introduce the necessary background on gradings and the objects that will be our main tools: composition algebras, cyclic composition algebras and trialitarian algebras. Section~\ref{s:Type_I} discusses Type I gradings, especially the facts that are relevant to the above-mentioned subtlety in their classification and will be crucial in Section~\ref{s:lifting}, where we show how to reduce the classification of Type III gradings on $\cL$ (or, equivalently, on the trialitarian algebra $E$) to the corresponding cyclic composition algebra. A description of Type III gradings on the latter is given in Section~\ref{s:Type_III_fine}, from where we derive the classification of fine gradings on $\cL$ up to equivalence, under the assumption $\chr\FF\ne 2$, thus extending the result known for characteristic $0$. In Section~\ref{s:Type_III}, we obtain the classification of Type III gradings by a fixed group $G$ up to isomorphism, which is new even for characteristic $0$. (If the reader is only interested in characteristic $0$, then there is no need to deal with affine group schemes: it is sufficient to consider the algebraic groups $\Aut_\FF(\cU)$ and $\wh{G}=\Hom(G,\FF^\times)$, which are the $\FF$-points of the schemes $\AAut_\FF(\cU)$ and $G^D$, respectively.) An appendix on graded modules is included with a twofold purpose: to complete the results in \cite{EK14} on graded modules for the classical simple Lie algebras by including Type III gradings for $D_4$, and to show the analogy of the Brauer invariants introduced in \cite{EK14} with Tits algebras.

\section{Preliminaries}\label{s:preliminaries}

\subsection{Group gradings on algebras}

Let $\cU$ be an algebra (not necessarily associative) over a field $\FF$ and let $G$ be a group (written multiplicatively).

\begin{df}\label{df:G_graded_alg}
A {\em $G$-grading} on $\cU$ is a vector space decomposition
\[
\Gamma:\;\cU=\bigoplus_{g\in G} \cU_g
\]
such that
$
\cU_g \cU_h\subset \cU_{gh}\quad\mbox{for all}\quad g,h\in G.
$
If such a decomposition is fixed, $\cU$ is referred to as a {\em $G$-graded algebra}.
The nonzero elements $x\in\cU_g$ are said to be {\em homogeneous of degree $g$}, and one writes $\deg_\Gamma x=g$ or just $\deg x=g$ if the grading is clear from the context. The {\em support} of $\Gamma$ is the set $\supp\Gamma\bydef\{g\in G\;|\;\cU_g\neq 0\}$.
\end{df}

If $(\cU,\sigma)$ is an algebra with involution, then we will always assume $\sigma(\cU_g)=\cU_g$ for all $g\in G$.

There is a more general concept of grading: a decomposition $\Gamma:\;\cU=\bigoplus_{s\in S}\cU_s$ into nonzero subspaces indexed by a set $S$ and having the property that, for any $s_1,s_2\in S$ with $\cU_{s_1}\cU_{s_2}\ne 0$, there exists (unique) $s_3\in S$ such that $\cU_{s_1}\cU_{s_2}\subset\cU_{s_3}$. For such a decomposition $\Gamma$, there may or may not exist a group $G$ containing $S$ that makes $\Gamma$ a $G$-grading. If such a group exists, $\Gamma$ is said to be a {\em group grading}. However, $G$ is usually not unique even if we require that it should be generated by $S$. The {\em universal grading group} is generated by $S$ and has the defining relations $s_1s_2=s_3$ for all $s_1,s_2,s_3\in S$ such that $0\ne\cU_{s_1}\cU_{s_2}\subset\cU_{s_3}$ (see e.g. \cite[Chapter 1]{EKmon} for details).

Here we will deal exclusively with abelian groups, and we will sometimes write them additively. Gradings by abelian groups often arise as eigenspace decompositions with respect to a family of commuting diagonalizable automorphisms. If $\FF$ is algebraically closed and $\chr\FF=0$ then all abelian group gradings on finite-dimensional algebras can be obtained in this way. Over an arbitrary field, a $G$-grading $\Gamma$ on $\cU$ is equivalent to a morphism of affine group schemes $\eta_\Gamma:G^D\to\AAut_\FF(\cU)$ as follows: for any $\cR\in\Alg_\FF$, the corresponding homomorphism of groups 
$(\eta_\Gamma)_\cR:\Alg_\FF(\FF G,\cR)\to\Aut_\cR(\cU\ot\cR)$ is defined by
\begin{equation}\label{eq:def_eta}
(\eta_\Gamma)_\cR(f)(x\ot r)=x\ot f(g)r\;\text{for all}\;x\in\cU_g, g\in G, r\in\cR, f\in\Alg_\FF(\FF G,\cR).
\end{equation}
Consequently, if we have two algebras, $\cU$ and $\cV$, and a morphism $\theta\colon\AAut_\FF(\cU)\to\AAut_\FF(\cV)$ then any $G$-grading $\Gamma$ on $\cU$ gives rise to a $G$-grading $\theta(\Gamma)$ on $\cV$ by setting $\eta_{\theta(\Gamma)}\bydef\theta\circ\eta_\Gamma$. 

Let $\Gamma:\, \cU=\bigoplus_{g\in G} \cU_g$ and $\Gamma':\,\cU'=\bigoplus_{h\in H} \cU'_h$
be two group gradings, with supports $S$ and $T$, respectively.
We say that $\Gamma$ and $\Gamma'$ are {\em equivalent} if there exists an isomorphism of algebras $\vphi\colon\cU\to\cU'$ and a bijection $\alpha\colon S\to T$ such that $\varphi(\cU_s)=\cU'_{\alpha(s)}$ for all $s\in S$. If $G$ and $H$ are universal grading groups then $\alpha$ extends to an isomorphism $G\to H$. In the case $G=H$, the 
$G$-gradings $\Gamma$ and $\Gamma'$ are {\em isomorphic} if $\cU$ and $\cU'$ are isomorphic as $G$-graded algebras, i.e., if there exists an isomorphism of algebras $\vphi\colon\cU\to\cU'$ such that $\varphi(\cU_g)=\cU'_g$ for all $g\in G$. Note that $\theta\colon\AAut_\FF(\cU)\to\AAut_\FF(\cV)$ sends isomorphic gradings on $\cU$ to isomorphic gradings on $\cV$.

If $\Gamma:\,\cU=\bigoplus_{g\in G} \cU_g$ and $\Gamma':\,\cU=\bigoplus_{h\in H} \cU'_h$ are two gradings on the same algebra, with supports $S$ and $T$, respectively, then we will say that $\Gamma'$ is a {\em refinement} of $\Gamma$ (or $\Gamma$ is a {\em coarsening} of $\Gamma'$) if for any $t\in T$ there exists (unique) $s\in S$ such that $\cU'_t\subset\cU_s$. If, moreover, $\cU'_t\ne\cU_s$ for at least one $t\in T$, then the refinement is said to be {\em proper}. A grading $\Gamma$ is said to be {\em fine} if it does not admit any proper refinement.

Given a $G$-grading $\Gamma:\,\cU=\bigoplus_{g\in G} \cU_g$, any group homomorphism $\alpha\colon G\to H$ induces an $H$-grading ${}^\alpha\Gamma$ on $\cU$ whose homogeneous component of degree $h$ is the sum of all $\cU_g$ with $\alpha(g)=h$. Note that $\eta_{{}^\alpha\Gamma}=\eta_\Gamma\circ\alpha^D$ and hence $\theta({}^\alpha\Gamma)={}^\alpha\theta(\Gamma)$. Clearly, ${}^\alpha\Gamma$ is a coarsening of $\Gamma$ (not necessarily proper). If $G$ is the universal group of $\Gamma$ then every coarsening of $\Gamma$ is obtained in this way. If $\Gamma$ and $\Gamma'$ are two gradings, with universal groups $G$ and $H$, then $\Gamma'$ is equivalent to $\Gamma$ if and only if $\Gamma'$ is isomorphic to ${}^\alpha\Gamma$ for some group isomorphism $\alpha\colon G\to H$. It follows that, if universal groups are used, $\theta$ preserves equivalence of gradings.

\subsection{Graded division algebras over algebraically closed fields}

Let $G$ be a group. If $\cR$ is an associative algebra with a $G$-grading such that $\cR$ is graded simple and satisfies the descending chain condition on graded left ideals then, by the graded version of a classical result, there exists a graded division algebra $\cD$ over $\FF$ and a finite-dimensional graded right vector space $W$ over $\cD$ such that $\cR\cong\End_\cD(W)$ as a $G$-graded algebra (\cite[Theorem 2.6]{EKmon}). Here by a {\em graded division algebra} we mean a unital associative algebra with a $G$-grading such that every nonzero homogeneous element is invertible. A {\em graded vector space} over $\cD$ is just a graded $\cD$-module (in the obvious sense), which is automatically free.

Now we collect for future use some general facts about finite-dimensional graded division algebras in the following situation: $\FF$ is algebraically closed and $G$ is abelian. 

Let $\cD$ be a graded division algebra and let $T$ be the support of the grading on $\cD$. Then $T$ is a subgroup of $G$ and $\cD$ can be identified with a twisted group algebra $\FF^\tau T$ for some $2$-cocycle $\tau\colon T\times T\to\FF^\times$. Indeed, the identity component $\cD_e$ is a division algebra over $\FF$, so $\cD_e=\FF$ and hence all nonzero homogeneous components have dimension $1$. We fix a basis $X_t$ in each of them and write $X_s X_t=\tau(s,t)X_{st}$ ($s,t\in T$). Define $\beta=\beta_\tau$ by $\beta(s,t)\bydef\frac{\tau(s,t)}{\tau(t,s)}$. This is an alternating bicharacter $T\times T\to\FF^\times$, independent of the scaling of the $X_t$ since $X_s X_t=\beta(s,t)X_t X_s$.  

The pair $(T,\beta)$ determines $\cD$ up to isomorphism of graded algebras. Indeed, the graded division algebras with support $T$ are classified by the cohomology class $[\tau]\in \mathrm{H}^2(T,\FF^\times)$ (see e.g. \cite[Theorem 2.13]{EKmon}), and we can use a standard cohomological result: the quotient of the group of symmetric $2$-cocycles $\mathrm{Z}^2_{sym}$ by the $2$-coboundaries $\mathrm{B}^2$ can be identified with $\mathrm{Ext}(T,\FF^\times)$ in the category of abelian groups, but $\FF^\times$ is a divisible group, so $\mathrm{Ext}(T,\FF^\times)$ is trivial, i.e., $\mathrm{B}^2=\mathrm{Z}^2_{sym}$. Since the mapping $\tau\mapsto\beta_\tau$ from $\mathrm{Z}^2$ to the group of alternating bicharacters, $\Hom(T\wedge T,\FF^\times)$, is a homomorphism with kernel $\mathrm{Z}^2_{sym}$, we obtain an injection $[\tau]\mapsto\beta_\tau$ from $\mathrm{H}^2(T,\FF^\times)$ to $\Hom(T\wedge T,\FF^\times)$ (in fact, an isomorphism, see \cite{Yamazaki}).

We will need the $G$-graded Brauer group of $\FF$, which we will denote by $B_G(\FF)$. There are several versions of Brauer group associated to a field (or, more generally, a commutative ring) $\FF$ and an abelian group $G$. They consist of equivalence classes of certain $\FF$-algebras equipped with a $G$-grading, a $G$-action or both. The multiplication is induced by tensor product of algebras or its twisted version. The Brauer group we need here is the one defined in \cite{PP}, where there is only a $G$-grading and the tensor product is not twisted. For a field $\FF$ and an abelian group $G$, the group $B_G(\FF)$ consists of the equivalence classes of finite-dimensional central simple associative $\FF$-algebras equipped with a $G$-grading, where $\cA_1\sim\cA_2$ if and only if there exist finite-dimensional $G$-graded $\FF$-vector spaces $V_1$ and $V_2$ such that $\cA_1\otimes\End(V_1)\cong\cA_2\otimes\End(V_2)$ as graded algebras. Every class contains a unique graded division algebra (up to isomorphism). The classical Brauer group $B(\FF)$ is imbedded in $B_G(\FF)$ as the classes containing a division algebra with trivial $G$-grading. As shown in \cite{PP}, if $G$ is finite and $\FF$ contains enough roots of unity (so that $|\wh{G}|=|G|$) then there is a split short exact sequence
$
1\to B(\FF)\to B_G(\FF)\to \mathrm{H}^2(\wh{G},\FF^\times)\to 1
$. If $\FF$ is algebraically closed (as we will assume in this paper) then, for any abelian group $G$, the Brauer group $B_G(\FF)$ is isomorphic to the group of alternating continuous bicharacters of the pro-finite group $\wh{G_0}$ where $G_0$ is the torsion subgroup of $G$ if $\chr{\FF}=0$ and the $p'$-torsion subgroup of $G$ if $\chr{\FF}=p>0$ (i.e., the set of all elements whose order is finite and coprime with $p$). Namely, the class of a $G$-graded matrix algebra corresponds to its ``commutation factor'' $\hat{\beta}$, which can be seen as an alternating bicharacter $\wh{G}\times \wh{G}\to\FF^\times$ (by means of the canonical homomorphism $\wh{G}\to\wh{G_0}$) and determines the parameters $(T,\beta)$ of the graded division algebra representing the class (see \cite[\S 2]{EK14}). Note that in this case $\beta$ is nondegenerate since the graded division algebra is central.

Let $\cK$ be the center of a graded division algebra $\cD$, so $\cK$ is a graded field over $\FF$. Let $H\subset T$ be the support of the grading on $\cK$. The following result holds for arbitrary $\FF$.

\begin{lemma}\label{lm:not_divisible}
Let $\cK$ be a graded field that is finite-dimensional over its identity homogeneous component $\FF$. If $\cK$ is separable as an algebra over $\FF$ then $\dim_\FF\cK$ is not divisible by $\chr{\FF}$.
\end{lemma}

\begin{proof}
Assume, to the contrary, that $\chr{\FF}=p>0$ and $p$ divides $\dim_\FF\cK=|H|$. Then $H$ contains an element $g$ of order $p$. Pick a nonzero element $x\in\cK_g$. Then $x^p\in\cK_e$ and the powers $1, x, x^2,\ldots,x^{p-1}$ are linearly independent over $\FF$, hence the minimal polynomial of $x$ has the form $X^p-\lambda$, $\lambda\in\FF$, which contradicts separability.
\end{proof}

Let $\wb{G}=G/H$ and $\wb{T}=T/H$. Since $H$ is precisely the radical of the alternating bicharacter $\beta$, we obtain a nondegenerate alternating bicharacter $\bar{\beta}\colon\wb{T}\times\wb{T}\to\FF^\times$, hence $|\wb{T}|$ is not divisible by $\chr{\FF}$. Assume $\cD$ is semisimple as an ungraded algebra. Then $\cK$ is the direct product of $k$ copies of $\FF$ where, by Lemma \ref{lm:not_divisible}, $k=|H|$ is not divisible by $\chr{\FF}$. 
Accordingly, we can write $\cD=\cD_1\times\cdots\times\cD_k$ where $\cD_i$ are the ideals generated by the minimal central idempotents of $\cD$, so they are simple and $\wb{G}$-graded. Since $\chr{\FF}$ does not divide $|H|$, the $H$-grading on $\cK$ is equivalent to an action of the group of characters $\wh{H}=\{\chi_1,\ldots,\chi_k\}$, which must permute the factors of $\cK$ transitively since $\cK$ is a graded field. Relabeling if necessary, we may assume that $\chi_i$ sends the first factor to the $i$-th one. We can extend the characters $\chi_i$ to $T$ in some way. Then each $\chi_i$ acts as an automorphism of $\cD$ that preserves the $G$-grading on $\cD$, hence $\chi_i$ yields an isomorphism of $\wb{G}$-graded algebras $\cD_1\to\cD_i$.

Since the identity component of each $\cD_i$ (with respect to the $\wb{G}$-grading) has dimension $1$, they are graded division algebras (see e.g. \cite[Lemma 2.20]{EKmon}). Therefore, we can regard them as elements of the $\wb{G}$-graded Brauer group $B_{\wb{G}}(\FF)$.

\begin{proposition}\label{prop:beta_bar}
Let $\cD$ be a graded division algebra over an algebraically closed field $\FF$ of arbitrary characteristic. Assume that $\cD$ is finite-dimensional and semisimple as an algebra, decomposing into simple components $\cD_1\times\cdots\times\cD_k$. Define $\wb{G}$, $\wb{T}$ and $\bar{\beta}$ as above. Then the $\cD_i$ are $\wb{G}$-graded division algebras, all isomorphic to each other, and the $\wb{G}$-graded Brauer class of $\cD_i$ corresponds to the pair $(\wb{T},\bar{\beta})$.
\end{proposition}

\begin{proof}
The only statement that has not yet been proved is the one about the Brauer class. Indeed, $\beta$ is defined by the equation $X_sX_t=\beta(s,t)X_tX_s$ for all $s,t\in T$. Since the value of $\beta$ depends only on the cosets of $s$ and $t$ in $\wb{T}$, in the $\wb{G}$-grading we have $xy=\bar{\beta}(\bar{s},\bar{t})yx$ for all $x\in\cD_{\bar{s}}$ and $y\in\cD_{\bar{t}}$. Taking the projection onto $\cD_i$, we see that $\bar{\beta}$ satisfies the same property with respect to the $\wb{G}$-grading of $\cD_i$ as $\beta$ with respect to the $G$-grading of $\cD$.
\end{proof}

\subsection{Composition algebras}

Recall that a (finite-dimensional) {\em composition algebra} is a nonassociative algebra $\cA$ with a nonsingular quadratic form $n$ such that $n(xy)=n(x)n(y)$ for all $x,y\in\cA$. It is known that $\dim\cA$ can be $1$, $2$, $4$ or $8$. The unital composition algebras are called {\em Hurwitz algebras} and can be obtained using the Cayley--Dickson doubling process. The ones of dimension $8$ are called {\em octonion}, or {\em Cayley algebras}. If the ground field is algebraically closed then, up to isomorphism, there is only one Hurwitz algebra in each dimension. 
Here we will also need another kind of composition algebras.

\begin{df}\label{df:symmetric}
A composition algebra $\cS$, with multiplication $\star$ and norm $n$, is said to be \emph{symmetric} if the polar form of the norm, $n(x,y)\bydef n(x+y)-n(x)-n(y)$, is associative:
\[
n(x\star y,z)=n(x,y\star z),
\]
for all $x,y,z\in\cS$.
\end{df}

As a consequence of this definition, $\cS$ satisfies the following identities:
\[
(x\star y)\star x=n(x)y=x\star(y\star x).
\]
Over an algebraically closed field, there are, up to isomorphism, only two symmetric composition algebras of dimension $8$: the para-Cayley and the Okubo algebras (see \cite{EPI96} or \cite[Theorem 4.44]{EKmon}). Both can be obtained from the Cayley algebra $\cC$ by introducing a new product: in the first case, $x\star y=x\bullet y\bydef \bar{x}\bar{y}$, where  $\bar{x}\bydef n(x,1)1-x$ is the standard involution, and in the second case, $x\star y=\tau(x)\bullet\tau^2(y)$, where $\tau$ is a certain automorphism of order $3$. If $\chr\FF\ne 3$ then the Okubo algebra can also be realized as the space of traceless $3\times 3$ matrices with a certain product (see e.g. \cite[\S 4.6]{EKmon} for details).

From now on, we assume $\chr{\FF}\ne 2$.

\begin{df}\label{df:tri}
Let $\cS$ be a symmetric composition algebra of dimension $8$. Its \emph{triality Lie algebra} is defined as
\[
\tri(\cS,\star,n)=\{(d_1,d_2,d_3)\in\frso(\cS,n)^3\;|\; d_1(x\star y)=d_2(x)\star y+x\star d_3(y)\ \forall x,y\in\cS\}.
\]
This is a Lie algebra with componentwise multiplication.
\end{df}

It turns out that this definition is symmetric with respect to cyclic permutations of $(d_1,d_2,d_3)$, and each projection determines an isomorphism $\tri(\cS)\to\frso(\cS,n)$, so $\tri(\cS)$ is a Lie algebra of type $D_4$ (see e.g. \cite[\S 5.5, \S 6.1]{EKmon} or \cite[\S 45.A]{KMRT}, but note that in the latter the ordering of triples differs from ours). 
This fact is known as the ``local triality principle''. There is also a ``global triality principle'', as follows.

\begin{df}\label{df:Tri}
Let $\cS$ be a symmetric composition algebra of dimension $8$. Its \emph{triality group} is defined as
\[
\TRI(\cS,\star,n)=\{(f_1,f_2,f_3)\in\Ort(\cS,n)^3\;|\; f_1(x\star y)=f_2(x)\star f_3(y)\ \forall x,y\in\cS\}.
\]
This is an algebraic group with componentwise multiplication.
\end{df}

It turns out that this definition is symmetric with respect to cyclic permutations of $(f_1,f_2,f_3)$, and $\TRI(\cS)$ is isomorphic to $\Spin(\cS,n)$. In fact, this isomorphism can be defined at the level of the corresponding group schemes (see \cite[\S 35.C]{KMRT} for details). The said cyclic permutations determine outer actions of $A_3$ on $\Spin(\cS,n)$ and its Lie algebra $\frso(\cS,n)$. If $\cS$ is a para-Cayley algebra then one can define an outer action of $S_3$ using $(f_1,f_2,f_3)\mapsto(\bar{f}_1,\bar{f}_3,\bar{f}_2)$ and $(d_1,d_2,d_3)\mapsto(\bar{d}_1,\bar{d}_3,\bar{d}_2)$ as the action of the transposition $(2,3)$, where $\bar{f}$ is defined by $\bar{f}(\bar{x})=\overline{f(x)}$. This allows us to split the exact sequence \eqref{eq:exact_D4}.

\subsection{Cyclic composition algebras}

A convenient way to ``package'' triples of maps as above is the following concept due to Springer. Let $\LL$ be a Galois algebra over $\FF$ with respect to the cyclic group of order $3$ (see e.g. \cite[\S 18.B]{KMRT}). Fix a generator $\rho$ of this group. For any $\ell\in\LL$, define the norm $N(\ell)=\ell\rho(\ell)\rho^2(\ell)$, the trace $T(\ell)=\ell+\rho(\ell)+\rho^2(\ell)$, and the adjoint $\ell^\#=\rho(\ell)\rho^2(\ell)$. Our main interest is in the case $\LL=\FF\times\FF\times\FF$ where $\rho(\ell_1,\ell_2,\ell_3)=(\ell_2,\ell_3,\ell_1)$.

\begin{df}\label{df:cyclic_comp}
A {\em cyclic composition algebra} over $(\LL,\rho)$ is a free $\LL$-module $V$ with a nonsingular $\LL$-valued quadratic form $Q$ and an $\FF$-bilinear multiplication $(x,y)\mapsto x*y$ that is $\rho$-semilinear in $x$ and $\rho^2$-semilinear in $y$ and satisfies the following identities:
\begin{align*}
Q(x*y)&=\rho(Q(x))\rho^2(Q(y)),\\
b_Q(x*y,z)&=\rho(b_Q(y*z,x))=\rho^2(b_Q(z*x,y)),
\end{align*}
where $b_Q(x,y)\bydef Q(x+y)-Q(x)-Q(y)$ is the polar form of $Q$.
An {\em isomorphism} from $(V,\LL,\rho,*,Q)$ to $(V',\LL',\rho',*',Q')$ is a pair of $\FF$-linear isomorphisms 
$\varphi_0\colon(\LL,\rho)\to(\LL',\rho')$ (i.e., $\varphi_0$ is an isomorphism that satisfies $\varphi_0\rho=\rho'\varphi_0$) and $\varphi_1\colon V\to V'$ such that $\varphi_1$ is $\varphi_0$-semilinear, $\varphi_1(x*y)=\varphi_1(x)*'\varphi_1(y)$ and $\varphi_0(Q(x))=Q'(\varphi_1(x))$ for all $x,y\in V$.
\end{df}

As a consequence, $V$ also satisfies
\begin{equation}\label{eq:cyclic_comp_id}
(x*y)*x=\rho^2(Q(x))y\quad\text{and}\quad x*(y*x)=\rho(Q(x))y.
\end{equation}
One checks that, for $\lambda,\mu\in\LL^\times$, the new product $x\mathbin{\tilde{*}}y=\lambda(x*y)$ and the new quadratic form $\tilde{Q}(x)=\mu Q(x)$ define a cyclic composition algebra if and only if $\mu=\lambda^\#$ (\cite[Lemma 36.1]{KMRT}). We will say that an isomorphism $(V,\LL,\rho,\tilde{*},\tilde{Q})\to(V',\LL',\rho',*',Q')$ is a {\em similitude} from $(V,\LL,\rho,*,Q)$ to $(V',\LL',\rho',*',Q')$ with {\em parameter} $\lambda$ and {\em multiplier} $\lambda^\#$. In particular, for any $\ell\in\LL^\times$, the mappings $\varphi_0=\id$ and $\varphi_1(x)=\ell x$ define a similitude from $(V,\LL,\rho,*,Q)$ to itself with parameter $\ell^{-1}\ell^\#$ and multiplier $\ell^2$.

If $(\cS,\star,n)$ is a symmetric composition algebra then $\cS\ot\LL$ becomes a cyclic composition algebra with $Q(x\ot\ell)=n(x)\ell^2$ (extended to sums in the obvious way using the polar form of $n$) and $(x\ot\ell)*(y\ot m)=(x\star y)\ot\rho(\ell)\rho^2(m)$. With $\LL=\FF\times\FF\times\FF$ and $\rho(\ell_1,\ell_2,\ell_3)=(\ell_2,\ell_3,\ell_1)$, this gives $V=\cS\times\cS\times\cS$ with $Q=(n,n,n)$ and
\begin{equation}\label{df:cyclic_prod}
(x_1,x_2,x_3)*(y_1,y_2,y_3)=(x_2\star y_3,x_3\star y_1,x_1\star y_2).
\end{equation}
It turns out that, with $\LL$ and $\rho$ as above, any cyclic composition algebra is similar to $\cS\ot\LL$ where $\cS$ is a para-Hurwitz algebra (see e.g. \cite[\S 36, \S 36B]{KMRT}). Hence, for any $\LL$, the $\LL$-rank of a cyclic composition algebra can be $1$, $2$, $4$ or $8$. Also, if $\FF$ is algebraically closed then there is only one isomorphism class of cyclic composition algebras in each rank. (In this case, similar cyclic composition algebras are isomorphic because, for any $\lambda\in\LL^\times$, we can find $\ell\in\LL^\times$ such that $\ell^{-1}\ell^\#=\lambda$.)

Assume now that the rank is $8$. The multiplication \eqref{df:cyclic_prod} allows us to interpret $\TRI(\cS,\star,n)$ as $\Aut_\LL(V,*,Q)$, the group of $\LL$-linear automorphisms (i.e., with $\varphi_0=\id$), and $\TRI(\cS,\star,n)\rtimes A_3$ as $\Aut_\FF(V,\LL,\rho,*,Q)$, the group of all automorphisms (see Definition \ref{df:cyclic_comp}). This interpretation can be made at the level of group schemes:
\[
\AAut_\FF(V,\LL,\rho,*,Q)=\TRIs(\cS,\star,n)\rtimes\mathbf{A}_3\cong\Spins(\cS,n)\rtimes\mathbf{A}_3.
\]
Similarly, $\tri(\cS,\star,n)$ can be interpreted as $\Der_\LL(V,*,Q)$.

\begin{remark}
If $\FF$ is algebraically closed then the para-Cayley and Okubo algebras are ``$\FF$-forms'' of the unique cyclic composition algebra $V$ of rank $8$ (in the sense that, for $\cS$ either the para-Cayley or Okubo algebra, we have $V\cong\cS\ot\LL$ as an $\LL$-module with a quadratic form, while the product is extended by semilinearity). Any outer automorphism of order $3$ defines such an ``$\FF$-form'' of $V$, which must be isomorphic either to the para-Cayley or the Okubo algebra; this explains the fact that there are $2$ conjugacy classes of outer automorphisms of the algebraic group $\Spin_8(\FF)$.
\end{remark}

We will later need the group scheme of similitudes $\SIMs_\FF(V,\LL,\rho,*,Q)$, which is defined as follows. 
Let $\Gs$ be the group scheme of invertible $\FF$-linear endomorphisms of $V$ that are semilinear with respect to some automorphism of $(\LL,\rho)$. This is a subgroupscheme of $\GLs_\FF(V)\times\AAut_\FF(\LL,\rho)$ whose projection onto the second factor can be split by choosing an ``$\FF$-form'' of $V$, thus identifying $\Gs$ with $\GLs_\LL(V)\rtimes\AAut_\FF(\LL,\rho)$. 

\begin{remark}
Here $\GLs_\LL(V)$ is regarded as a group scheme over $\FF$, sending $\cR\in\Alg_\FF$ to the group of invertible elements in $\End_{\LL\ot\cR}(V\ot\cR)$. One could define $\GLs_\LL(V)$ as a group scheme over $\LL$, sending $\cR'\in\Alg_\LL$ to the group of invertible elements in $\End_{\cR'}(V\ot_\LL\cR')$. Our $\GLs_\LL(V)$ is the $\LL/\FF$-corestriction of this latter (i.e., the group scheme over $\FF$ obtained by substituting $\LL\ot\cR$ for $\cR'$). In particular, we will regard the multiplicative group scheme of $\LL$ as a group scheme over $\FF$ and will denote it by $\GLs_1(\LL)$.
\end{remark}

Now, we have natural representations of $\Gs$ in the $\FF$-linear space of $\LL$-bilinear functions $V\times V\to\LL$ and in the $\FF$-linear space of $\FF$-bilinear functions $V\times V\to V$ that are $\rho$-semilinear in the first variable and $\rho^2$-semilinear in the second variable. Then $\AAut_\FF(V,\LL,\rho,*,Q)$ is the intersection of the stabilizers in $\Gs$ of the vectors $b_Q$ and $*$. We define $\SIMs_\FF(V,\LL,\rho,*,Q)$ as the intersection of the stabilizers of the $\LL$-submodules spanned by each of $b_Q$ and $*$.

Using $\FF b_Q$ as an $\FF$-form of the free $\LL$-module $\LL b_Q$, we identify the group scheme of semilinear endomorphisms of $\LL b_Q$ with $\GLs_\LL(\LL b_Q)\rtimes\AAut_\FF(\LL,\rho)=\GLs_1(\LL)\rtimes\mathbf{A}_3$, and similarly for $\LL *$. Thus the representations of $\SIMs_\FF(V,\LL,\rho,*,Q)$ on $\LL *$ and on $\LL b_Q$ give rise to morphisms $\SIMs_\FF(V,\LL,\rho,*,Q)\to\GLs_1(\LL)\rtimes\mathbf{A}_3$, which we denote by $\theta$ and $\theta'$, respectively. We will now obtain a relation between these two morphisms, to be used later. For any $\cR\in\Alg_\FF$, we extend the action of $\rho$ from $\LL$ to $\LL\ot\cR$ by $\cR$-linearity and define $N$ and $\#$ for $\LL\ot\cR$ by the same formulas as before. Hence we obtain an automorphism of the group scheme $\GLs_1(\LL)$ defined by $a\mapsto a^\#$ for all $a\in(\LL\ot\cR)^\times$, $\cR\in\Alg_\FF$. Together with the identity on $\mathbf{A}_3$, this yields an automorphism of $\GLs_1(\LL)\rtimes\mathbf{A}_3$, which we will denote by $(\#,\id)$.

\begin{lemma}\label{lm:lambda_mu}
We have $\theta'=(\#,\id)\theta$.
\end{lemma}

\begin{proof}
We fix $\cR\in\Alg_\FF$, denote $\LL_\cR=\LL\ot\cR$ and $V_\cR=V\ot\cR$, and extend $b_Q$ and $*$ to $V_\cR$ by $\cR$-linearity. Let $(\varphi_1,\varphi_0)$ be an element of $\SIMs_\FF(V,\LL,\rho,*,Q)(\cR)\subset\Gs(\cR)$, so $\varphi_1$ is an invertible $\varphi_0$-semilinear endomorphism of the $\LL_\cR$-module $V_\cR$. For any $a\in\LL_\cR$, the action of $(\varphi_1,\varphi_0)$ sends $ab_Q$ to the bilinear form $V_\cR\times V_\cR\to\LL_\cR$ defined by $(x,y)\mapsto\varphi_0(aQ(\varphi_1^{-1}(x),\varphi_1^{-1}(y)))$, which must be an element of $\LL_\cR b_Q$. We see that this mapping on $\LL_\cR b_Q$ is $\varphi_0$-semilinear and hence $\theta'_\cR(\varphi_1,\varphi_0)=(\mu,\varphi_0)\in (\LL_\cR)^\times\rtimes\Aut_\cR(\LL_\cR,\rho)$, where $\mu b_Q$ is the image of $b_Q$ under the action. Similarly, $\theta_\cR(\varphi_1,\varphi_0)=(\lambda,\varphi_0)$, where $\lambda *$ is the image of $*$. By construction, $(\varphi_1,\varphi_0)$ is an isomorphism $(V_\cR,\LL_\cR,\rho,*,Q)\to(V_\cR,\LL_\cR,\rho,\tilde{*},\tilde{Q})$ where $\tilde{Q}(x)=\mu Q(x)$ and $x\mathbin{\tilde{*}}y=\lambda(x*y)$. 
It follows that $\mu=\lambda^\#$.
\end{proof}

\subsection{The trialitarian algebra $\End_\LL(V)$}

Let $V$ be a cyclic composition algebra over $(\LL,\rho)$ of rank $8$. Consider $E=\End_\LL(V)$. This is a central separable associative algebra over $\LL$. We will denote by $\sigma$ the involution of $E$ determined by the quadratic form $Q$. The even Clifford algebra $\Cl_0(V,Q)$ can be defined purely in terms of $(E,\sigma)$ as the quotient $\Cl(E,\sigma)$ of the tensor algebra of $E$ (regarded as an $\LL$-module) by certain relations --- see \cite[\S 8.B]{KMRT}, where the construction is carried out for a central simple algebra over a field, but we can apply it to each simple factor of $E$ if $\LL=\FF\times\FF\times\FF$. The imbedding of $E$ into its tensor algebra yields a canonical $\LL$-linear map $\kappa\colon E\to\Cl(E,\sigma)$, whose image generates $\Cl(E,\sigma)$, but which is neither injective nor a homomorphism of algebras. This construction  has the advantage of being functorial for isomorphisms of algebras with involution. Thus, $\sigma$ determines a unique involution $\ul{\sigma}$ on $\Cl(E,\sigma)$ such that $\kappa\sigma=\ul{\sigma}\kappa$. Also, any action (respectively, grading) by a group on the algebra with involution $(E,\sigma)$ gives rise to a unique action (respectively, grading) on the algebra with involution $(\Cl(E,\sigma),\ul{\sigma})$ such that $\kappa$ is equivariant (respectively, preserves the degree). In fact, we can define a morphism of group schemes $\AAut_\FF(E,\sigma)\to\AAut_\FF(\Cl(E,\sigma),\ul{\sigma})$. There is an isomorphism $\Cl_0(V,Q)\to\Cl(E,\sigma)$ sending, for any $x,y\in V$, the element $x\cdot y\in\Cl_0(V,Q)$ to the image of the operator $z\mapsto xb_Q(y,z)$ under $\kappa$.

\begin{df}\label{df:tri_alg}
The multiplication $*$ of $V$ allows us to define an additional structure on $E$, namely, an isomorphism of $\LL$-algebras with involution:
\[
\alpha\colon\Cl(E,\sigma)\stackrel{\sim}{\to}\,{}^{\rho} E\times{}^{\rho^2}E,
\]
where the superscripts denote the twist of scalar multiplication (i.e., ${}^\rho E$ is $E$ as an $\FF$-algebra with involution, but with the new $\LL$-module structure defined by $\ell\cdot a=\rho(\ell)a$). This is done using the Clifford algebra $\Cl(V,Q)$ as follows. Identities \eqref{eq:cyclic_comp_id} imply that the mapping
\[
x\mapsto\begin{pmatrix}0 & l_x \\ r_x & 0\end{pmatrix}\in\End_{\LL}({}^\rho V\oplus{}^{\rho^2}V),\;x\in V,
\]
where $l_x(y)\bydef x*y\defby r_y(x)$, extends to an isomorphism of $\ZZ_2$-graded algebras with involution
\[
\alpha_V\colon(\Cl(V,Q),\tau)\stackrel{\sim}{\to}(\End_{\LL}({}^\rho V\oplus{}^{\rho^2}V),\tilde\sigma),
\]
where $\tau$ is the standard involution of the Clifford algebra and $\tilde\sigma$ is induced by the quadratic form $({}^\rho Q,{}^{\rho^2}Q)$ on ${}^\rho V\oplus{}^{\rho^2}V$, with ${}^\rho Q(x)\bydef\rho^{-1}(Q(x))$. Then $\alpha$ is obtained by restricting $\alpha_V$ to the even part $\Cl_0(V,Q)$ and identifying $\Cl(E,\sigma)$ with $\Cl_0(V,Q)$ as above. Explicitly,
\[
\alpha\colon \kappa\bigl(xb_Q(y,\cdot)\bigr)\mapsto(l_x r_y,r_x l_y),\;x,y\in V.
\]
\end{df}

The structure $(E,\LL,\rho,\sigma,\alpha)$ is an example of {\em trialitarian algebra} (with trivial discriminant) --- the general definition is given in \cite[\S 43.A]{KMRT} but we will not need it here. An {\em isomorphism} $(E,\LL,\rho,\sigma,\alpha)\to(E',\LL',\rho',\sigma',\alpha')$ is defined to be an isomorphism $\varphi\colon (E,\sigma)\to (E',\sigma')$ of $\FF$-algebras with involution such that the following diagram commutes:
\[
\xymatrix{
\Cl(E,\sigma)\ar[r]^\alpha\ar[d]_{\Cl(\varphi)} & {}^{\rho}E\times{}^{\rho^2}E\ar[d]^{\varphi\ot\Delta(\varphi)}\\
\Cl(E',\sigma')\ar[r]^{\alpha'} & {}^{\rho'}E'\times{}^{\rho'^2}E'
}
\]
where we have identified ${}^\rho E\times{}^{\rho^2}E=E\ot(\FF\times\FF)$ as $\FF$-algebras (similarly for $E'$) and $\Delta(\varphi)\in\Aut_\FF(\FF\times\FF)$ is the identity if the restriction $\varphi\colon\LL\to\LL'$ conjugates $\rho$ to $\rho'$ and the flip if it conjugates $\rho$ to the inverse of $\rho'$. Note that this definition does not depend on the identifications of $E$ and $E'$ with algebras of endomorphisms but depends on the choice of ``orientations'' $\rho$ and $\rho'$. Extending scalars to $\cR\in\Alg_\FF$, one defines the automorphism group scheme $\AAut_\FF(E,\LL,\sigma,\alpha)$, which does not depend on the choice of $\rho$. It can be shown (\cite[\S 44]{KMRT}) that, for $V=\cC\ot\LL$ with $\LL=\FF\times\FF\times\FF$ and $(\cC,\bullet,n)$ a para-Cayley algebra, the group scheme $\AAut_\FF(E,\LL,\sigma,\alpha)$ is isomorphic to $\PGOs^+(\cC,n)\rtimes\mathbf{S}_3$. Moreover, the natural morphism $\mathrm{Int}\colon\Gs\to\AAut_\FF(E)$, where $\Gs$ is the group scheme of invertible $\FF$-linear endomorphisms of $V$ that are semilinear over $(\LL,\rho)$ and $\mathrm{Int}_\cR(a)$ is the conjugation by $a\in\Gs(\cR)$ on $E_\cR=\End_{\LL\ot\cR}(V\ot\cR)$, $\cR\in\Alg_\FF$, yields a morphism of short exact sequences:
\[
\xymatrix{
\mathbf{1}\ar[r] & \AAut_\LL(V,\LL,\rho,*,Q)\ar[r]\ar[d] & \AAut_\FF(V,\LL,\rho,*,Q)\ar[r]\ar[d] & 
\AAut_\FF(\LL,\rho)\ar[r]\ar[d] & \mathbf{1}\\
\mathbf{1}\ar[r] & \AAut_\LL(E,\LL,\sigma,\alpha)\ar[r] & \AAut_\FF(E,\LL,\sigma,\alpha)\ar[r] & 
\AAut_\FF(\LL)\ar[r] & \mathbf{1}
}
\]
where the left vertical arrow is the quotient map $\Spins(\cC,n)\to\PGOs^+(\cC,n)$, the middle vertical arrow has image $\AAut_\FF(E,\LL,\rho,\sigma,\alpha)$ (i.e., the automorphisms that preserve $\rho$), and the right vertical arrow is the injection $\mathbf{A}_3\to\mathbf{S}_3$.

Our interest in the trialitarian algebra $\End_\LL(V)$ comes from the following fundamental connection with Lie algebras of type $D_4$ (see \cite[\S 45]{KMRT}). It turns out that the restriction $\frac12\kappa\colon\mathrm{Skew}(E,\sigma)\to\mathrm{Skew}(\Cl(E,\sigma),\ul{\sigma})$ is an injective homomorphism of Lie algebras over $\LL$, and the $\FF$-subspace
\[
\cL(E,\LL,\rho,\sigma,\alpha)\bydef\{x\in\mathrm{Skew}(E,\sigma)\;|\;\alpha(\kappa(x))=2(x,x)\}
\]
is precisely the Lie subalgebra $\tri(\cC,\bullet,n)=\Der_\LL(V,*,Q)$, which is isomorphic to $\So(\cC,n)$ by the local triality principle. (The equality $\cL(E,\LL,\rho,\sigma,\alpha)=\tri(\cC,\bullet,n)$ follows from the well-known fact that $\tri(\cC,\bullet,n)$ is spanned by the triples of the form $\bigl(xn(y,\cdot)-yn(x,\cdot),\frac12(r_xl_y-r_yl_x),\frac12(l_xr_y-l_yr_x)\bigr)$, $x,y\in\cC$, where $l_x(y)\bydef x\bullet y\defby r_y(x)$ --- see e.g. \cite[\S 5.5]{EKmon}, where this fact is used to prove the local triality principle.) Moreover, the restriction $\mathrm{Res}$ from $E$ to 
$\cL(E)\bydef \cL(E,\LL,\rho,\sigma,\alpha)$ allows us to recover the short exact sequence \eqref{eq:exact_D4} as follows:
\[
\xymatrix{
\mathbf{1}\ar[r] & \AAut_\LL(E,\LL,\sigma,\alpha)\ar[r]\ar[d]^[@!-90]{\sim} & 
\AAut_\FF(E,\LL,\sigma,\alpha)\ar[r]\ar[d]^[@!-90]{\sim} & \AAut_\FF(\LL)\ar[r]\ar@{=}[d] & \mathbf{1}\\
\mathbf{1}\ar[r] & \AAut_\FF(\cL(E))_0\ar[r] & \AAut_\FF(\cL(E))\ar[r] & 
\AAut_\FF(\LL)\ar[r] & \mathbf{1}\\
}
\]
where the subscript $0$ denotes the connected component of identity. Note that the composition $\mathrm{Res}\circ\mathrm{Int}$ is the adjoint representation 
\[
\Ad\colon\AAut_\FF(V,\LL,\rho,*,Q)\to\AAut_\FF(\Der(V,*,Q)). 
\]
Since there is only one trialitarian algebra (up to isomorphism) over an algebraically closed field, it follows that, over any field, each Lie algebra of type $D_4$ can be realized as $\cL(E)$ for a unique trialitarian algebra (though not necessarily of the sort discussed here). And what is important for our purpose, any $G$-grading on $\cL(E)$ is the restriction of a unique $G$-grading on $E$, so the classification of gradings is the same for $\cL(E)$ as for $E$.

\section{Type I gradings and their Brauer invariants}\label{s:Type_I}

Let $\FF$ be an algebraically closed field, $\chr\FF\ne 2$. Take $\LL=\FF\times\FF\times\FF$ and $\rho\in\Aut_\FF(\LL)$ defined by $\rho(\lambda_1,\lambda_2,\lambda_3)=(\lambda_2,\lambda_3,\lambda_1)$.

Let $(V,\LL,\rho,*,Q)$ be a cyclic composition algebra of rank $8$ and let $E=\End_\LL(V)$ be the corresponding trialitarian algebra. Suppose we have a grading $\Gamma$ on $(E,\LL,\rho,$ $\sigma,\alpha)$ by an abelian group $G$ (which we may assume finitely generated), i.e., a grading on the $\FF$-algebra $E$ such that $\sigma$ and $\alpha$ preserve the degree. 

In this section we assume that $\Gamma$ is of Type I, i.e, its restriction to the center $\LL=Z(E)$ is trivial. Then the decomposition $\LL=\FF\times\FF\times\FF$ yields the $G$-graded decomposition $E=E_1\times E_2\times E_3$, where each $E_i$ is isomorphic to $M_8(\FF)$ and equipped with an orthogonal involution $\sigma_i$ (the restriction of $\sigma$). Denote the $G$-grading on $E_i$ by $\Gamma_i$, $i=1,2,3$. 

\subsection{Related triples of gradings on $M_8(\FF)$ with orthogonal involution}

Write $V=\cC\otimes\LL$ where $(\cC,\bullet,n)$ is the para-Cayley algebra. Then we can identify each $(E_i,\sigma_i)$ with $\End_\FF(\cC)$, where the 
involution is induced by the norm $n$. Thus, we can view all $\Gamma_i$ as gradings on the same algebra $\End_\FF(\cC)\cong M_8(\FF)$. We will say that $(\Gamma_1,\Gamma_2,\Gamma_3)$ is the {\em related triple} associated to $\Gamma$. 

Recall that the $G$-grading $\Gamma$ is equivalent to a morphism $G^D\to\AAut_\LL(E,\LL,\sigma,\alpha)$, where $G^D$ is the Cartier dual of $G$. The triality principle implies that each restriction morphism $\pi_i\colon\AAut_\LL(E,\LL,\sigma,\alpha)\to\AAut_\FF(E_i,\sigma_i)$ is a closed imbedding whose image is the connected component of $\AAut_\FF(E_i,\sigma_i)$ (isomorphic to $\PGOs_8^+$). Hence each of the gradings $\Gamma_i=\pi_i(\Gamma)$ uniquely determines $\Gamma$.
Moreover, there exists $\Gamma$ such that $\Gamma_1$ is any given ``inner'' grading on the algebra with involution $\End_\FF(\cC)$. 

The outer action of $S_3$ on $\TRIs(\cC,\bullet,n)$ and on its quotient $\AAut_\LL(E,\LL,\sigma,\alpha)$ yields the following action 
on related triples: $A_3$ permutes the components of $(\Gamma_1,\Gamma_2,\Gamma_3)$ cyclically and the transposition $(2,3)$ sends 
$(\Gamma_1,\Gamma_2,\Gamma_3)$ to $(\wb{\Gamma}_1,\wb{\Gamma}_3,\wb{\Gamma}_2)$ where $\wb{\Gamma}_i$ denotes the image of $\Gamma_i$ under the inner automorphism of $\End_\FF(\cC)$ corresponding to 
the standard involution of $\cC$ (which is an improper isometry of $n$).

The classification up to isomorphism of $G$-gradings on the algebra $M_8(\FF)$ with orthogonal involution is known (see e.g. \cite[Theorem 2.64]{EKmon}). It can also be determined explicitly which of these gradings are ``inner'' (see Remark \ref{rem:inner}). Let us see to what extent this allows us to classify Type I gradings on the trialitarian algebra $E$ or, equivalently, on the Lie algebra $\cL(E)$. For two ``inner'' gradings $\Gamma'$ and $\Gamma''$, we will write $\Gamma'\sim\Gamma''$ if there exists an element of $\PGO_8^+(\FF)$ sending $\Gamma'$ to $\Gamma''$. Thus, $\Gamma'$ and $\Gamma''$ are isomorphic if and only if $\Gamma'\sim\Gamma''$ or $\wb{\Gamma'}\sim\Gamma''$.

The stabilizer of a given Type I grading $\Gamma$ under the outer action of $S_3$ can have size $1$, $2$, $3$ or $6$. 
Therefore, we have the following three possibilities (the last one corresponding to sizes $3$ and $6$):
\begin{itemize}
\item If $\Gamma_i\nsim\Gamma_j$ for $i\ne j$ and $\Gamma_i\nsim\wb{\Gamma}_j$ for all $i,j$ then the isomorphism class of $\Gamma$ corresponds to $3$ distinct isomorphism classes of ``inner'' gradings (one for each of $\Gamma_i$);
\item If $\Gamma_i\sim\wb{\Gamma}_i$ for some $i$ and $\Gamma_i\nsim\Gamma_j\sim\wb{\Gamma}_k$, where $\{i,j,k\}=\{1,2,3\}$, then the isomorphism class of $\Gamma$ corresponds to $2$ distinct isomorphism classes of ``inner'' gradings (one for $\Gamma_i$ and one for $\Gamma_j$);
\item If $\Gamma_1\sim\Gamma_2\sim\Gamma_3$ then the isomorphism class of $\Gamma$ corresponds to $1$ isomorphism class of ``inner'' gradings.  
\end{itemize}
Given $\Gamma_1$, one can --- in principle --- compute $\Gamma_2$ and $\Gamma_3$, but obtaining explicit formulas seems to be difficult.

\subsection{Relation among the Brauer invariants in a related triple}

Recall that the projections $E\to E_i$ define the three irreducible representations of dimension $8$ for the Lie algebra $\cL=\cL(E)$ of type $D_4$: the natural and the two half-spin representations. As before, suppose we are given a Type I grading $\Gamma$ on $E$ and consider its related triple $(\Gamma_1,\Gamma_2,\Gamma_3)$. Since the image of $\cL$ generates each $E_i$, the gradings $\Gamma_i$ on $E_i$ are the unique $G$-gradings such that the three representations $\cL\to E_i$ are  homomorphisms of graded Lie algebras. 

For each $i$, we can write $E_i=\End_{\cD_i}(W_i)$ where $\cD_i$ is a graded division algebra and $W_i$ is a graded right vector space over $\cD_i$ (\cite[Theorem 2.6]{EKmon}). We are interested in the Brauer classes $[E_i]$, i.e., the isomorphism classes of $\cD_i$, $i=1,2,3$.

Since the even Clifford algebra corresponding to the space of the natural representation of $\cL$ can be defined purely in terms of the corresponding algebra with involution $E_i$ (say, $i=1$) and hence inherits the $G$-grading regardless of the characteristic of $\FF$, we can derive relations among the $G$-graded Brauer classes $[E_i]$ using the same arguments as in \cite[Proposition 39]{EK14}: 
\begin{equation}\label{eq:Brauer_relations}
[E_i]^2=1\text{ and }[E_1]=[E_2][E_3]. 
\end{equation}
Alternatively, extending the proof of \cite[Theorem 9.12]{KMRT} to the graded setting, one arrives at the same relations (see also Proposition 42.7 therein).

\section{Lifting Type III gradings from $\End_\LL(V)$ to $V$}\label{s:lifting}

Let $\FF$ be an algebraically closed field, $\chr\FF\ne 2,3$, and $\LL=\FF\times\FF\times\FF$. Define $\rho\in\Aut_\FF(\LL)$ by $\rho(\lambda_1,\lambda_2,\lambda_3)=(\lambda_2,\lambda_3,\lambda_1)$ and fix a primitive cube root of unity $\omega\in\FF$. The element $\xi=(1,\omega,\omega^2)$ spans the $\omega$-eigenspace for $\rho$ in $\LL$ and satisfies $N(\xi)=1$.

Let $(V,\LL,\rho,*,Q)$ be a cyclic composition algebra of rank $8$ and let $E=\End_\LL(V)$ be the corresponding trialitarian algebra. Suppose we have a Type III grading $\Gamma$ on $(E,\LL,\rho,$ $\sigma,\alpha)$ by an abelian group $G$, i.e, the projection of the image of $\eta=\eta_\Gamma\colon G^D\to\AAut_\FF(E,\LL,\sigma,\alpha)$ in $\AAut_\FF(\LL)$ is $\mathbf{A}_3$. This is equivalent to saying that the $G$-grading on $\LL$ obtained by restricting $\Gamma$ to the center $\LL=Z(E)$ has 1-dimensional homogeneous components: $\LL=\FF\oplus\FF\xi\oplus\FF\xi^2$, so $\LL$ becomes a graded field. Let $h\in G$ be the degree of $\xi$. This is an element of order $3$, which we will call the {\em distinguished element} of the grading. Note that the subgroup $H\bydef\langle h\rangle$ is precisely the distinguished subgroup, i.e., the subgroup of $G$ corresponding to the inverse image $\eta^{-1}(\eta(G^D)\cap\AAut_\LL(E,\LL,\sigma,\alpha))$ in $G^D$.

The purpose of this section is to show that $\eta$ can be ``lifted'' from $E$ to $V$ in the sense that there exists a morphism $\eta'$ such that the following diagram commutes:

\begin{equation}
\begin{aligned}\label{eq:eta_diag}
\xymatrix{
& \AAut_\FF(V,\LL,\rho,*,Q)\ar[dd]^{\mathrm{Int}}\\
G^D\ar@{-->}[ru]^-{\eta'}\ar[rd]^-{\eta} &\\
& \AAut_\FF(E,\LL,\sigma,\alpha)
}
\end{aligned}
\end{equation}
or, in other words, there exists a $G$-grading $\Gamma'$ on $(V,\LL,\rho,*,Q)$, i.e., a grading on the $\FF$-algebra $(V,*)$ making it a graded $\LL$-module with $b_Q$ a degree-preserving map, such that the grading on $E=\End_\LL(V)$ induced by $\Gamma'$ is precisely $\Gamma$.

We will proceed in steps. First we will forget about all extra structure and will treat $E$ simply as a central algebra over $\LL$, so $V$ will be just a graded vector space over $\LL$. Then we will include the involution $\sigma$ and the bilinear form $Q$ in the picture, and finally will take care of $\alpha$ and $*$.

\subsection{Triviality of the graded Brauer invariants}\label{sse:lifting1}

Since $E$ is graded simple, we can write $E=\End_\cD(W)$ as a $G$-graded algebra, where $\cD$ is a graded division algebra over $\FF$ and $W$ is a graded right vector space over $\cD$ (\cite[Theorem 2.6]{EKmon}). Clearly, $\LL$ is the center of $\cD$. We are going to show that $\cD=\LL$. 

Let $\wb{G}=G/H$ and consider the $\wb{G}$-grading on $E$ induced by the quotient map $G\to\wb{G}$. Since it is a Type I grading, we can decompose  
$E=E_1\times E_2\times E_3$ as a $\wb{G}$-graded algebra. We also have $\cD=\cD_1\times\cD_2\times\cD_3$, $W=W_1\times W_2\times W_3$ and $E_i=\End_{\cD_i}(W_i)$. 
Since the $E_i$ are simple algebras, so are the $\cD_i$. Now Proposition \ref{prop:beta_bar} tells us that the $\wb{G}$-graded algebras $\cD_1$, $\cD_2$ and $\cD_3$ are isomorphic graded division algebras, hence $[E_1]=[E_2]=[E_3]$ in the $\wb{G}$-graded Brauer group.
On the other hand, we have relations \eqref{eq:Brauer_relations}. This forces $[E_i]=1$ for all $i$. In other words, $\wb{T}$ and $\bar{\beta}$ are trivial, 
which means that $T=H$ and $\beta=1$. We have proved $\cD=\LL$.

\subsection{Construction of $\eta'$}

Since the graded division algebra corresponding to the $G$-graded algebra $E$ is the graded field $\LL$, we can give $V$ a $G$-grading $\Gamma'$ making it into a graded vector space over $\LL$ such that the grading induced on $E=\End_\LL(V)$ is precisely the given grading $\Gamma$. But so far we ignored the quadratic form $Q$ and the multiplication $*$ on $V$. Taking $Q$ into account is easy: by replacing $Q$ with $\mu Q$ for a suitable $\mu\in\LL^\times$, we can make $b_Q\colon V\times V\to\LL$ a homogeneous map of some degree  and the possible choices of $\mu$ differ by a homogeneous factor in $\LL^\times$ (\cite[Theorem 2.57]{EKmon}). Since $\FF$ is algebraically closed, we can find $\lambda\in\LL^\times$ such that $\lambda^\#=\mu$. Replacing $*$ with $\lambda*$ and $Q$ with $\mu Q$, we obtain a structure of cyclic composition algebra on $V$ (similar to the original one) that has homogeneous $b_Q$ and induces the original $\sigma$ and $\alpha$ on $E$. We will now show that, by a suitable shift of grading $\Gamma'$, we can make $*$ and $Q$ degree-preserving (i.e., homogeneous maps of degree $e$).

Recall the morphism $\mathrm{Int}\colon\Gs\to\AAut_\FF(E)$, where $\Gs$ is the group scheme of invertible $\FF$-linear endomorphisms of $V$ that are semilinear over $(\LL,\rho)$. The grading $\Gamma'$ on $V$ is equivalent to a morphism $\eta'\colon G^D\to\Gs$ such that $\mathrm{Int}\circ\eta'=\eta$. Since $\eta(G^D)$ is contained in $\mathbf{H}\bydef\AAut_\FF(E,\LL,\rho,\sigma,\alpha)$, we conclude that $\eta'(G^D)$ is contained in $\Gs_0\bydef\mathrm{Int}^{-1}(\mathbf{H})$. Since $\mathrm{Int}$ is a separable morphism (i.e., the kernel is smooth), $\Gs_0$ is smooth as the inverse image of a smooth subgroupscheme. On the other hand, $\mathrm{Int}(\SIMs(V,\LL,\rho,*,Q))$ is contained in $\mathbf{H}$, hence $\SIMs(V,\LL,\rho,*,Q)$ is contained in $\Gs_0$. Since the $\FF$-points of these two group schemes coincide, the schemes have the same dimension, hence $\SIMs(V,\LL,\rho,*,Q)$ is also smooth and coincides with $\Gs_0$. We have shown that $\eta'(G^D)$ is contained in $\SIMs(V,\LL,\rho,*,Q)$. This means that $\eta'(G^D)$ stabilizes the $\LL$-submodules spanned by each of $b_Q$ and $*$.  

Now recall the morphisms $\theta'$ and $\theta$ of $\SIMs(V,\LL,\rho,*,Q)$ to $\GLs_1(\LL)\rtimes\mathbf{A}_3$ associated to the actions of this group scheme on the $\LL$-submodules $\LL b_Q$ and $\LL*$, respectively. We know that $b_Q$ is a homogeneous element in the space of bilinear forms (with respect to the grading induced by $\Gamma'$), hence $\eta'(G^D)$ stabilizes the $\FF$-subspace spanned by $b_Q$. This means that, for any $\cR\in\Alg_\FF$, the homomorphism $\theta'_\cR$ sends every element of the group $\eta'(G^D)(\cR)$ to an element of the form $(\mu,\varphi_0)\in(\LL_\cR)^\times\rtimes\Aut_\cR(\LL_\cR,\rho)$, where $\mu$ actually belongs to $\cR^\times$. Let $(\lambda,\varphi_0)$ be the image of the same element of $\eta'(G^D)(\cR)$ under $\theta_\cR$. In view of Lemma \ref{lm:lambda_mu}, we have $\mu=\lambda^\#$. But then $N(\lambda)\lambda=\mu^\#\in\cR^\times$, hence also $\lambda\in\cR^\times$. We have shown that $\eta'(G^D)$ stabilizes the $\FF$-subspace spanned by $*$. In terms of the grading $\Gamma'$, this means that $*$ is a homogeneous element of some degree, say, $g_0$, i.e., we have $x*y\in V_{g_0ab}$ for all $x\in V_a$ and $y\in V_b$ ($a,b\in G$). Shifting the grading $\Gamma'$ by $g_0$, we obtain $V_a*V_b\subset V_{ab}$, i.e., the new $\eta'$ sends $G^D$ to the stabilizer of $*$. Applying Lemma \ref{lm:lambda_mu} again, we see that $\eta'(G^D)$ stabilizes $b_Q$ as well. (Alternatively, one may invoke identities \eqref{eq:cyclic_comp_id} relating $*$ and $Q$.) We have constructed $\eta'$ that fits diagram \eqref{eq:eta_diag} for the cyclic composition algebra similar to the original one.

\subsection{Reduction of the classification of Type III gradings from trialitarian algebras to cyclic composition algebras}

First we summarize the result of the previous subsection:

\begin{theorem}\label{th:tri_to_comp}
Let $(E,\LL,\rho,\sigma,\alpha)$ be a trialitarian algebra over an algebraically closed field $\FF$, $\chr{\FF}\ne 2,3$. Suppose $E$ is given a Type III grading by an abelian group $G$. Then there exists a cyclic composition algebra $(V,\LL,\rho,*,Q)$ with a $G$-grading such that $E$ is isomorphic to $\End_\LL(V)$ as a $G$-graded trialitarian algebra.\qed  
\end{theorem}

Since we want to classify $G$-gradings up to isomorphism, there still remains the question of uniqueness of $V$ in the above theorem. To state the answer, it is convenient to introduce some terminology. For a cyclic composition algebra $(V,\LL,\rho,*,Q)$ over $(\LL,\rho)$, consider the same $\LL$-module $V$ with the same quadratic form $Q$ but with the new multiplication $x*^\mathrm{op}y\bydef y*x$. This is a cyclic composition algebra over $(\LL,\rho^2)$, called the {\em opposite} of $V$ and denoted $V^\mathrm{op}$. We will say that an isomorphism $(\varphi_1,\varphi_0)\colon(V,\LL,\rho)\to(V',\LL',\rho')$ and an algebra isomorphism $\varphi\colon\End_\LL(V)\to\End_{\LL'}(V')$ are {\em compatible} if $\varphi_1(ax)=\varphi(a)\varphi_1(x)$ for all $a\in\End_\LL(V)$ and $x\in V$. The next result does not require algebraic closure.

\begin{theorem}\label{th:tri_isomorphism}
Let $(V,\LL,\rho,*,Q)$ and $(V',\LL',\rho',*',Q')$ be two cyclic composition algebras of rank $8$ where $\LL=\FF\times\FF\times\FF$ such that $\chr{\FF}\ne 2,3$ and $\FF$ contains a primitive cube root of unity. Suppose $V$ and $V'$ are given Type III gradings by an abelian group $G$. Let $E=\End_\LL(V)$ and $E'=\End_{\LL'}(V')$ be the corresponding trialitarian algebras with induced $G$-gradings. Then, for any isomorphism $(\varphi_1,\varphi_0)$ of graded cyclic composition algebras from either $V$ or $V^\mathrm{op}$ to $V'$, there exists a unique compatible isomorphism  $\varphi\colon E\to E'$ of graded trialitarian algebras. Conversely, for any isomorphism $\varphi\colon E\to E'$ of graded trialitarian algebras, there exists a unique compatible isomorphism of graded cyclic composition algebras $(\varphi_1,\varphi_0)$ from either $V$ or $V^\mathrm{op}$ (but not both) to $V'$. 
\end{theorem}

\begin{proof}
Given $(\varphi_1,\varphi_0)$, the only mapping $\varphi\colon E\to E'$ that will satisfy the compatibility condition is the one defined by $\varphi(a)x'=\varphi_1(a\varphi_1^{-1}(x'))$, for all $a\in E$ and $x'\in V'$, and it is an isomorphism of algebras. Since $\sigma$ and $\alpha$ are defined in terms of $Q$ and $*$, it is straightforward to verify that $\varphi$ is actually an isomorphism of trialitarian algebras. The definition of induced grading on the algebra of endomorphisms of a graded module implies that $\varphi$ preserves the degree.

Conversely, suppose $\varphi$ is given. Let $\varphi_0\colon\LL\to\LL'$ be the restriction of $\varphi$. (Since $V'$ is a faithful $\LL'$-module, this is the only possibility to satisfy the compatibility condition.) Then either $\varphi_0\colon(\LL,\rho)\to(\LL',\rho')$ or $\varphi_0\colon(\LL,\rho^2)\to(\LL',\rho')$. The second possibility reduces to the first if we replace $V$ with $V^\mathrm{op}$, so assume $\varphi_0\colon(\LL,\rho)\to(\LL',\rho')$. By \cite[Proposition 44.16]{KMRT}, there exists a similitude of (ungraded) cyclic composition algebras $(\tilde{\varphi}_1,\tilde{\varphi}_0)\colon(V,\LL,\rho,*,Q)\to(V',\LL',\rho',*',Q')$. Let $\tilde{\varphi}$ be the corresponding isomorphism $E\to E'$. Then $\tilde{\varphi}^{-1}\varphi$ is an automorphism of $(E,\LL,\rho,\sigma,\alpha)$, and it follows from \cite[Proposition 44.2]{KMRT} that the group $\Aut_\FF(E,\LL,\rho,\sigma,\alpha)$ is the image of the homomorphism $\mathrm{Int}_\FF\colon\SIM_\FF(V,\LL,\rho,*,Q)\to\Aut_\FF(E,\LL,\sigma,\alpha)$. Therefore, we can find $(\psi_1,\psi_0)\in\SIM_\FF(V,\LL,\rho,*,Q)$ that is sent to $\tilde{\varphi}^{-1}\varphi$, in particular $\psi_0=\tilde{\varphi}_0^{-1}\varphi_0$. Set $\hat{\varphi}_1\bydef\tilde{\varphi}_1\psi_1$. Then $(\hat{\varphi}_1,\varphi_0)$ will satisfy the compatibility condition with $\varphi$. 

It remains to take care of the gradings. By \cite[Theorem 2.10]{EKmon}, there exists $u\in G$ and $\varphi_0$-semilinear isomorphism $\varphi_1\colon V^{[u]}\to V'$ of graded spaces over $\LL$ such that $(\varphi_1,\varphi_0)$ is compatible with $\varphi$. Here $V^{[u]}$ denotes a shift of grading, i.e., the new grading $V=\bigoplus_{g\in G}\tilde{V}_g$ where $\tilde{V}_{gu}=V_g$ for all $g\in G$. Now observe that $\hat{\varphi}_1^{-1}\varphi_1$ is an endomorphism of $V$ as an $E$-module, hence it is the multiplication by an element $\ell\in\LL^\times$, i.e., $\varphi_1(x)=\hat{\varphi}_1(\ell x)$ for all $x\in V$. Since $\hat{\varphi}_1$ is a similitude, so is $\varphi_1$. If $\hat{\lambda}\in\LL^\times$ is the parameter of $\hat{\varphi}_1$ then the parameter of $\varphi_1$ is $\lambda=\hat{\lambda}\ell^{-1}\ell^\#$. On the other hand, we have $\varphi_1(V_g)=V'_{gu}$ for all $g\in G$. Pick $s,t\in G$ with $V_s*V_t\ne 0$ and pick $x\in V_s$, $y\in V_t$ such that $z\bydef x*y\ne 0$. Then $z'\bydef\varphi_1(z)$ is a nonzero element of $V'_{stu}$. At the same time, we have $\varphi_0(\lambda)z'=\varphi_1(\lambda(x*y))=\varphi_1(x)*'\varphi(y)\in V'_{stu^2}$. It follows that $\varphi_0(\lambda)$ is a homogeneous element of degree $u$. Replacing $\varphi_1$ by the map $x\mapsto\varphi_1(\lambda^{-1}x)$, which is a similitude with parameter $\lambda_0\bydef\lambda^2(\lambda^\#)^{-1}\in\LL_e^\times=\FF^\times$, we obtain $\varphi_1(V_g)=V'_g$ for all $g\in G$. Finally, since $\lambda_0\in\FF^\times$, the mapping $x\mapsto\lambda_0 x$ is a similitude $V\to V$ with parameter $\lambda_0$ and leaves each $V_g$ invariant, hence replacing $\varphi_1$ by the map $x\mapsto\varphi_1(\lambda_0^{-1}x)$ yields the desired isomorphism of graded cyclic composition algebras.

Finally, if $(\tilde{\varphi}_1,\tilde{\varphi}_0)$ is another isomorphism compatible with $\varphi$ and preserving degree then there exists $\ell\in\LL^\times$ such that $\tilde{\varphi}_1(x)=\varphi_1(\ell x)$ for all $x\in V$. But this is possible only if $\ell=1$. 
\end{proof}

\begin{corollary}\label{cor:tri_isomorphism}
Under the conditions of Theorem \ref{th:tri_to_comp}, fix an identification $E=\End_\LL(V)$ as a $G$-graded algebra. Then there exist exactly $4$ gradings on the cyclic composition algebra $V$ that induce the given grading on $E$, and they form an orbit of the subgroup 
\[
C\bydef\{(\veps_1,\veps_2,\veps_3)\;|\;\veps_i\in\{\pm 1\},\,\veps_1\veps_2\veps_3=1\}\subset\LL^\times
\]
(the center of the spin group) with respect to its natural action on $V$.
\end{corollary}

\begin{proof}
If $V$ has two $G$-gradings, $\Gamma_1$ and $\Gamma_2$, that induce the same grading $\Gamma$ on $E$ then there exists an algebra automorphism $\varphi_1\colon V\to V$ that sends $\Gamma_1$ to $\Gamma_2$ and induces the identity map on $E$ (so $\varphi_0=\id$). Hence $\varphi_1$ is given by  $\varphi_1(x)=\ell x$ for some $\ell\in\LL^\times$. This map is a similitude with multiplier $\ell^{-1}\ell^\#$, which must be equal to $1$. It is easy to see that $\ell^{-1}\ell^\#=1$ if and only if $\ell\in C$. 

It remains to observe that, if $\ell\in C$ is different from $1$, then, for any homogeneous component $V_g\ne 0$ of $\Gamma_1$, the image $\varphi_1(V_g)=\ell V_g$ cannot coincide with $V_g$. Indeed, consider any character $\chi\in\wh{G}$ such that $\chi(h)=\omega$. In the action of the group $\wh{G}$ associated to the grading $\Gamma_1$ on $V$, $\chi$ will act differently on $V_g$ and on $\ell V_g$, namely, as the scalar $\chi(g)$ on the former and as  $\rho(\ell)\ell^{-1}\chi(g)$ on the latter.
\end{proof}

We now turn to the classification of fine gradings up to equivalence. Clearly, if a Type III grading cannot be refined in the class of Type III gradings then it is fine. It is also clear from Theorems \ref{th:tri_to_comp} and \ref{th:tri_isomorphism} that a fine Type III grading on a cyclic composition algebra $V$ induces a fine Type III grading on the trialitarian algebra $E=\End_\LL(V)$. We will now establish the converse.

\begin{theorem}\label{th:tri_equivalence}
Let $(E,\LL,\rho,\sigma,\alpha)$ be a trialitarian algebra over an algebraically closed field $\FF$, $\chr{\FF}\ne 2,3$. Then every fine Type III grading on $E$ with universal group $G$ is induced from a fine Type III grading on the cyclic composition algebra $(V,\LL,\rho,*,Q)$ with the same universal group. Moreover, two such gradings on $E$ are equivalent if and only if they are induced from equivalent gradings on $V$. 
\end{theorem}

\begin{proof}
Let $\Gamma$ be a fine Type III grading on $E$ with universal group $G$. It is determined by a maximal diagonalizable subgroupscheme $\Qs$ of $\AAut_\FF(E,\LL,\rho,\sigma,\alpha)$, which we may identify with $G^D$. By Theorem \ref{th:tri_to_comp}, we can induce $\Gamma$ from a Type III $G$-grading $\Gamma'$ on $V$. Let $\Qs'$ be the corresponding diagonalizable subgroupscheme of $\AAut_\FF(V,\LL,\rho,*,Q)$. Then the morphism $\mathrm{Int}$ restricts to an isomorphism $\Qs'\to\Qs$ (see diagram \eqref{eq:eta_diag}). We claim that $\Gamma'$ is fine. Assume, to the contrary, that $\Qs'$ is not maximal. Then there exists diagonalizable $\tilde{\Qs}'$ properly containing $\Qs'$. The image $\mathrm{Int}(\tilde{\Qs}')$ is necessarily $\Qs$ because the latter is maximal. We will obtain a contradiction if we can show that the intersection of $\tilde{\Qs}'$ with the kernel $\mathbf{K}$ of $\mathrm{Int}$ is trivial. Since $\mathbf{K}$ is isomorphic to $\bmu_2^2$ and every subgroupscheme of $\bmu_2^2$ is smooth, it suffices to show that the intersection has no $\FF$-points different from the identity. But this is clear since $\mathbf{K}(\FF)=C$ (see the Corollary \ref{cor:tri_isomorphism}) and $\Qs'(\FF)$ contains an  automorphism that acts as a cyclic permutation on $C$ and hence does not commute with any element of $C$ except the identity. The assertion about equivalence follows from Theorem \ref{th:tri_isomorphism}.
\end{proof}

\section{Gradings on cyclic composition algebras}\label{s:Type_III_fine}

As in the previous section, let $\FF$ be an algebraically closed field, $\chr\FF\ne 2,3$, and let $\LL=\FF\times\FF\times\FF$, with $\rho(\lambda_1,\lambda_2,\lambda_3)=(\lambda_2,\lambda_3,\lambda_1)$.

\subsection{The Albert algebra}

If $(V,\LL,\rho,*,Q)$ is a cyclic composition algebra of rank $8$, then the direct sum
\begin{equation}\label{eq:JLV}
\cJ(\LL,V)\bydef \LL\oplus V,
\end{equation}
is the Albert algebra (i.e., the simple exceptional Jordan algebra), which contains $\LL$ as a subalgebra and whose  
norm, trace form and adjoint extend those of $\LL$ in the following way (see \cite[Theorem 38.6]{KMRT}):
\begin{align*}
&N\bigl((\ell,v)\bigr)=N(\ell)+b_Q(v,v*v)-T\bigl(\ell Q(v)\bigr),\\
&T\bigl((\ell_1,v_1),(\ell_2,v_2)\bigr)= T(\ell_1\ell_2)+T\bigl(b_Q(v_1,v_2)\bigr),\\
&(\ell,v)^\#=\bigl(\ell^\#-Q(v),v*v-\ell v).
\end{align*}
In particular, $V$ is the orthogonal complement to $\LL$ relative to the trace form.

Any element $X$ in $\alb\bydef\cJ(\LL,V)$ satisfies the generic degree $3$ equation:
\[
X^3-T(X)X^2+S(X)X-N(X)1=0,
\]
where $T\bigl((\ell,v)\bigr)=T(\ell)$ and $S(X)=\frac{1}{2}\left(T(X)^2-T(X^2)\right)$. 
Note that the adjoint is defined by $X^\#=X^2-T(X)X+S(X)1$, hence the commutative multiplication in $\alb$ can be expressed in terms of $\#$ and $T$.
Therefore, any grading $\Gamma$ on the cyclic composition algebra $(V,\LL,\rho,*,Q)$ by an abelian group $G$ extends to a grading $\Gamma_\cJ$ on $\alb$ by the same group $G$, given by $\alb_g=\LL_g\oplus V_g$ for all $g\in G$. The gradings on the Albert algebra have been determined in \cite{EK12a} (see also \cite[Chapter 5]{EKmon}).

\subsection{Type III gradings on cyclic composition algebras}

Recall that, for a symmetric composition algebra $(\cS,\star,n)$, the associated cyclic composition algebra $(V,\LL,\rho,*,Q)$ is given by $V=\cS\ot \LL$, 
\begin{align*}
&(x\ot \ell)*(y\ot m)=(x\star y)\ot \rho(\ell)\rho^2(m),\\
&Q(x\otimes\ell)=n(x)\ell^2,\\
&b_Q(x\otimes\ell,y\otimes m)=n(x,y)\ell m,
\end{align*}
for all $x,y\in\cS$ and $\ell,m\in\LL$. If we think of $\cS\ot \LL$ as $\cS\times\cS\times\cS$ then the product expands as in \eqref{df:cyclic_prod}, namely, 
\[
(x_1,x_2,x_3)*(y_1,y_2,y_3)=(x_2\star y_3,x_3\star y_1,x_1\star y_2)
\]
and $Q$ becomes $(n,n,n)$. This cyclic composition algebra will be denoted by $(\cS,\star,n)\ot(\LL,\rho)$.

Any pair of gradings, $\Gamma_\cS$ on $\cS$ and $\Gamma_\LL$ on $\LL$ (such that $\rho$ is degree-preserving), by the same abelian group $G$, induces a $G$-grading $\Gamma$ on the cyclic composition algebra $(\cS,\star,n)\ot(\LL,\rho)$ with $\bigl(\cS\ot\LL)_g=\bigoplus_{k\in G}\left(\cS_{gk^{-1}}\ot \LL_k\right)$ for all $g\in G$. This grading will be denoted by $\Gamma_\cS\ot\Gamma_\LL$. We are interested in the case of Type III gradings, where $\LL$ is a graded field: its homogeneous components are $\LL_e=\FF 1$, $\LL_h=\FF\xi$ and $\LL_{h^2}=\FF\xi^2$ where, as before, $\xi=(1,\omega,\omega^2)$ and $h\in G$ is the distinguished element.

\begin{theorem}\label{th:gradings_cyclic}
Let $\Gamma$ be a Type III grading by an abelian group $G$ on the cyclic composition algebra $(V,\LL,\rho,*,Q)$ of rank $8$ over an algebraically closed field $\FF$, $\chr\FF\ne 2,3$, and let  $\Gamma_\LL$ be the induced grading on $\LL$. 
\begin{enumerate}
\item 
If $V_e=0$, then  $(V,\LL,\rho,*,Q)$ is isomorphic to $(\cO,\star,n)\ot(\LL,\rho)$ as a graded cyclic composition algebra, where $(\cO,\star,n)$ is the Okubo algebra, endowed with a $G$-grading $\Gamma_\cO$ with $\cO_e=0$, and the grading on $(\cO,\star,n)\ot(\LL,\rho)$ is $\Gamma_\cO\ot\Gamma_\LL$.
\item 
Otherwise, $(V,\LL,\rho,*,Q)$ is isomorphic to $(\cC,\bullet,n)\ot(\LL,\rho)$ as a graded cyclic composition algebra, where $(\cC,\bullet,n)$ is the para-Cayley algebra, endowed with a $G$-grading $\Gamma_\cC$, and the grading on $(\cC,\bullet,n)\ot(\LL,\rho)$ is $\Gamma_\cC\ot\Gamma_\LL$.
\end{enumerate}
\end{theorem}

\begin{proof}
Assume first that $V_e=0$ and consider the Albert algebra $\alb=\cJ(\LL,V)$ as in the previous subsection. Then the grading $\Gamma_\cJ$ induced by $\Gamma$ on $\alb$ satisfies the condition $\alb_e=\FF 1$, and hence, by \cite[Theorem 5.12]{EKmon}, all the homogeneous components of $\Gamma_\cJ$ have dimension $1$ and the support is a $3$-elementary abelian subgroup of $G$ (isomorphic to $\ZZ_3^3$). Set $h_1=h$ (the distinguished element) and pick $h_2,h_3\in G$ such that the support of $\Gamma_\cJ$ is generated by $h_1,h_2,h_3$. Then $V=\bigoplus_{g\in \supp\Gamma}V_g$, with $\dim V_g=1$ for any $g\in \supp\Gamma$, and $\supp\Gamma=\supp\Gamma_\cJ\setminus H=\langle h_1,h_2,h_3\rangle\setminus\langle h_1\rangle$.

Consider the graded $\FF$-subalgebra $\cO\bydef\bigoplus_{g\in\langle h_2,h_3\rangle} V_g$ in $(V,*)$. The values of 
$Q$ (or $b_Q$) on $\cO$ are contained in 
$\bigoplus_{g\in\langle h_2,h_3\rangle}\LL_g=\LL_e=\FF 1$,
and hence $n\bydef Q\vert_\cO$ is a nondegenerate quadratic form on $\cO$. Because of identity \eqref{eq:cyclic_comp_id}, $\cO$ is a symmetric composition algebra of dimension $8$ with norm $n$. Besides, it is graded by $\langle h_2,h_3\rangle\cong\ZZ_3^2$, with one-dimensional homogeneous components and support $\langle h_2,h_3\rangle\setminus\{e\}$. It follows that $(\cO,*,n)$ is the Okubo algebra (see \cite{Eld09} or \cite[Theorems 4.12 and 4.51]{EKmon}). Moreover, $V=\cO\oplus\xi\cO\oplus\xi^2\cO=\cO\ot\LL$ and the first part of the theorem follows.

We proceed to the case $V_e\ne 0$. Then $V_e$ is an $\FF$-subalgebra of $(V,*)$ and, again, the values of $Q$ on $V_e$ are contained in $\FF 1$. Hence $(V_e,*,Q)$ is a symmetric composition algebra. Since $\FF$ is algebraically closed, there exists a nonzero idempotent $\varepsilon$ in $V_e$: $0\ne\varepsilon=\varepsilon*\varepsilon$ (see e.g. \cite[Proposition 4.43]{EKmon}). Substituting $x=y=\varepsilon$ into identity \eqref{eq:cyclic_comp_id} gives $Q(\varepsilon)=1$.

The cyclic composition algebra $(V,\LL,\rho,*,Q)$ is isomorphic to $(\cC,\bullet,n)\ot(\LL,\rho)$, with  $(\cC,\bullet,n)$ the para-Cayley algebra, so we may identify these and hence we may identify $\varepsilon$ with a triple $(x_1,x_2,x_3)\in\cC\times\cC\times\cC$ such that $x_i\bullet x_{i+1}=x_{i+2}$ for any $i=1,2,3$ (indices modulo $3$) and $n(x_1)=n(x_2)=n(x_3)=1$. Using \cite[Corollary 5.6 and Lemma 5.25]{EKmon} we conclude that there is a triple $(f_1,f_2,f_3)\in \TRI(\cC,\bullet,n)$ such that $f_1(x_1)=f_2(x_2)=1$ (the unit of the Cayley algebra or, equivalently, the para-unit of the para-Cayley algebra). But then we get $f_3(x_3)=f_3(x_1\bullet x_2)=f_1(x_1)\bullet f_2(x_2)=1\bullet 1=1$. Since $\TRI(\cC,\bullet,n)$ is contained in the group of automorphisms of our cyclic composition algebra, we may assume, without loss of generality, that $\varepsilon=\buno\bydef (1,1,1)$.

Note that, for any $(x_1,x_2,x_3)\in V=\cC\ot \LL$, we have
\[
(x_1,x_2,x_3)*\buno=(\bar x_2,\bar x_3,\bar x_1)=b_Q\bigl((x_2,x_3,x_1),\buno\bigr)\buno-(x_2,x_3,x_1).
\]
For $X=(x_1,x_2,x_3)\in V$, define $\bar{X}\bydef b_Q(X,\buno)\buno-X$. Hence $X*\buno=\bar{X}$ if and only if $x_1=x_2=x_3$, if and only if $X\in \cC\ot 1$. But $\{X\in V\;|\; X*\buno=\bar{X}\}$ is a graded subspace of $V$, so we conclude that $\cC\cong \cC\ot 1$ (with the para-Hurwitz multiplication) is a graded $\FF$-subalgebra of $(V,*)$, and the second part of the Theorem follows.
\end{proof}

\subsection{Application to fine gradings on simple Lie algebras of type $D_4$}

Let $\cL$ be the simple Lie algebra of type $D_4$ over an algebraically closed field $\FF$, $\chr\FF\ne 2$. We are ready to obtain the classification of fine gradings on $\cL$ up to equivalence. Recall that Type III gradings exist only if $\chr\FF\ne 3$. In this case we use the realization $\cL=\cL(E)$ where $E$ is the trialitarian algebra. The following result implies the analogs of Theorems 6.8 and 6.11 in \cite{EKmon}, where only characteristic $0$ was considered.

\begin{corollary}\label{cor:fineIII}
Up to equivalence, there are three fine gradings of Type III on the simple Lie algebra of type $D_4$ over an algebraically closed field $\FF$, $\chr\FF\ne 2,3$.
Their universal groups are $\ZZ^2\times\ZZ_3$, $\ZZ_2^3\times\ZZ_3$ and $\ZZ_3^3$.
\end{corollary}

\begin{proof}
Recall that, since the restriction from $E$ to $\cL$ yields an isomorphism of automorphism group schemes, $E$ and $\cL$ have the same classification of gradings. By Theorem \ref{th:tri_equivalence}, the classification of fine gradings of Type III on $E$ is the same as that on $V$, the cyclic composition algebra of rank $8$. Finally, Theorem \ref{th:gradings_cyclic} implies that any fine grading of Type III on $V$ comes from a fine grading on either the para-Cayley algebra $\cC$ or the Okubo algebra $\cO$, with $\cO_e=0$. On $\cC$, there are only two fine gradings, up to equivalence; their universal groups are $\ZZ^2$ and $\ZZ_2^3$ (see e.g. \cite[Theorem 4.51]{EKmon}). On $\cO$, there is a unique fine grading with trivial identity component; its universal group is $\ZZ_3^2$ (see e.g. \cite[Corollary 4.54]{EKmon}). Conversely, these gradings on $\cC$ and $\cO$ give rise to three gradings on $V$, whose universal groups are $\ZZ^2\times\ZZ_3$, $\ZZ_2^3\times\ZZ_3$ and $\ZZ_3^3$. By looking at $V_e$ and the universal groups, we see that none of these three can be a coarsening of another, hence they are fine.
\end{proof}

We now turn to Types I and II (so $\chr\FF=3$ is allowed) and use the matrix realization $\cL=\mathrm{Skew}(\cR,\sigma)$ where $\cR=M_8(\FF)$ and $\sigma$ is an orthogonal involution. Up to equivalence, there are $15$ fine gradings on $(\cR,\sigma)$, which restrict to $15$ gradings of Type I or II on $\cL$ (see e.g. \cite[Example 3.44]{EKmon}). 

\begin{remark}\label{rem:inner}
Since we assume $\chr\FF\ne 2$, a $G$-grading on $(\cR,\sigma)$ restricts to a Type~I grading on $\cL$ if and only if the group of characters $\wh{G}$ acts by inner automorphisms of $\cL$. Hence Lemma 33 in \cite{EK14} allows us to determine which restrictions are of Type I and which of Type II and to compute the generator of the distinguished subgroup (see Definition 34 in \cite{EK14}). A direct computation shows that, out of the above $15$ gradings on $\cL$, $8$ are of Type I and $7$ are of Type II.
\end{remark}

\begin{theorem}\label{th:D4_fine}
Let $\FF$ be an algebraically closed field and let $\cL$ be the simple Lie algebra of type $D_4$ over $\FF$.
\begin{enumerate}
\item If $\chr\FF\ne 2,3$ then there are, up to equivalence, $17$ fine gradings on $\cL$. Their universal groups and types are given in Theorem 6.15 of \cite{EKmon}.
\item If $\chr\FF=3$ then there are, up to equivalence, $14$ fine gradings on $\cL$. They correspond to cases (1)---(14) in Theorem 6.15 of \cite{EKmon}.
\end{enumerate}
\end{theorem}

\begin{proof}
The Type II gradings on $\cL$ obtained from fine gradings on $(\cR,\sigma)$ cannot be refined in the class of Type II gradings, hence they are fine. The Type I gradings, on the other hand, could fail to be fine because of the possibility of a Type III refinement (if $\chr\FF\ne 3$). However, such a grading $\Gamma$ would then be the coarsening of one of the gradings in Corollary \ref{cor:fineIII} obtained by taking the universal group modulo the distinguished subgroup of order $3$, so the universal group of $\Gamma$ would be $\ZZ^2$, $\ZZ_2^3$ or $\ZZ_3^2$, but none of these occurs on the list (cf. \cite[Corollary 6.12]{EKmon}).

Out of the $15$ fine gradings on $\cL$ coming from $(\cR,\sigma)$, there are only two that share the same universal group (namely, $\ZZ_2^3\times\ZZ_4$) and the same type ($24$ components of dimension $1$ and $2$ components of dimension $2$) --- see the discussion in \cite[\S 6.1]{EKmon} following Corollary 6.12. It turns out that these two are actually equivalent. This can be shown as in \cite{EKmon}, since the proofs of Lemma 6.13 and Proposition 6.14 are valid under the assumption $\chr\FF\ne 2$. Alternatively, we can consider the graded Brauer invariants $(T_i,\beta_i)$ of the related triple $(\Gamma_1,\Gamma_2,\Gamma_3)$ where $\Gamma_1$ is one of these two gradings (which are of Type I). For one of them, we have $T_1\cong\ZZ_2^2$, so Equation (30) and Remark 43 in \cite{EK14} apply. For the other, we have $T_1\cong\ZZ_2^4$, so Equation (32) and Remark 44 apply. Whichever we choose, it is easy to see, using the results in \cite{EK14} mentioned above,  that the $T_i$ are not all isomorphic to each other (in fact, two of them are $\ZZ_2^4$ and one is $\ZZ_2^2$), hence the $\Gamma_i$ are not all equivalent as gradings on $M_8(\FF)$ and the result follows.  
\end{proof}

\section{Type III gradings up to isomorphism}\label{s:Type_III}

We continue to assume that the ground field $\FF$ is algebraically closed and that $\chr\FF\ne 2,3$. 

Let $\Gamma$ be a Type III grading by an abelian group $G$ on the cyclic composition algebra $(V,\LL,\rho,*,Q)$ of rank $8$. 
Define the \emph{rank} of $\Gamma$ as the dimension of the neutral homogeneous component $V_e$. Recall from the proof of Theorem \ref{th:gradings_cyclic} 
that either $V_e=0$, or $V_e$ is a symmetric composition algebra, and hence its dimension is restricted to $1$, $2$, $4$ or $8$.

Given two such Type III gradings $\Gamma$ and $\Gamma'$, they will be said to be \emph{similar} if they induce isomorphic gradings 
on the trialitarian algebra $E=\End_\LL(V)$ or, equivalently, on the Lie algebra $\cL(E)$, which is simple of type $D_4$. 
According to Theorem \ref{th:tri_isomorphism}, $\Gamma$ and $\Gamma'$ are similar if and only if the graded cyclic composition algebras $(V,\Gamma)$ and $(V,\Gamma')$
are isomorphic or anti-isomorphic, i.e., there exists an isomorphism $(\varphi_1,\varphi_0)$ from either $V$ or $V^\mathrm{op}$, endowed with the grading $\Gamma$,
onto $V$, endowed with the grading $\Gamma'$. Thus, the classification of $G$-gradings on the simple Lie algebra of type $D_4$ up to isomorphism is the same as 
the classification of $G$-gradings on the cyclic composition algebra $(V,\LL,\rho,*,Q)$ up to similarity.
Clearly, the rank is an invariant of the similarity class of a given Type III grading $\Gamma$.


\subsection{Construction of Type III gradings}
Fix an abelian group $G$. For each possible rank $r$, we will define a list of Type III gradings by $G$ on the unique cyclic composition algebra $(V,\LL,\rho,*,Q)$,
which can be realized as $(\cC,\bullet,n)\otimes(\LL,\rho)$ or as $(\cO,\star,n)\otimes(\LL,\rho)$, where $(\cC,\bullet,n)$ and $(\cO,\star,n)$ 
are the para-Cayley and the Okubo algebra, respectively. As before, $h$ will denote the distinguished element of the grading 
(the degree of $\xi=(1,\omega,\omega^2)\in\LL$, which spans the $\omega$-eigenspace of $\rho$), so the distinguished subgroup is $H=\langle h\rangle$. 
Note that the distinguished elements of $V$ and $V^\mathrm{op}$ are inverses of each other.

\begin{itemize}
\item[$\boxed{r=0}$] By \cite[Corollaries 4.54 and 4.55]{EKmon}, given a subgroup $K$ of $G$ isomorphic to $\ZZ_3^2$, there are, up to isomorphism, 
exactly two $G$-gradings, $\Gamma_\cO^+$ and $\Gamma_\cO^-$, with support $K\setminus\{e\}$ on the Okubo algebra $(\cO,\star,n)$. 
Pick an order $3$ element $h\in G\setminus K$ and let $\Gamma_\LL$ be the grading on $\LL$ with $\deg \xi=h$. 
Denote by $\GIII_{0}(G,K,h,\pm)$ the grading $\Gamma_\cO^\pm\otimes\Gamma_\LL$ on $(\cO,\star,n)\otimes(\LL,\rho)$. 
Note that the support of this grading is $KH\setminus H$ and the subgroup generated by the support is 
the direct product $KH$, where $H=\langle h\rangle$.

\item[$\boxed{r=1}$] By \cite[Theorems 4.21 and 4.51]{EKmon}, given a subgroup $K$ of $G$ isomorphic to $\ZZ_2^3$, there is, up to isomorphism, a unique grading $\Gamma_\cC$  with support $K$
on the Cayley algebra, or equivalently on the para-Cayley algebra $(\cC,\bullet,n)$. Pick an order $3$ element $h\in G\setminus K$ and let $\Gamma_\LL$ be as above.
Denote by $\GIII_{1}(G,K,h)$ the grading $\Gamma_\cC\otimes\Gamma_\LL$ on $(\cC,\bullet,n)\otimes(\LL,\rho)$.

\item[$\boxed{r=2}$] Pick an order $3$ element $h\in G$ and elements $g_1,g_2,g_3\in G\setminus H$, $H=\langle h\rangle$, with $g_1g_2g_3=e$. 
Write $\gamma=(g_1,g_2,g_3)$. Consider the grading $\Gamma_\cC(G,\gamma)$ on the para-Cayley algebra $(\cC,\bullet,n)$ induced from the Cartan grading 
(see \cite[Theorem 4.21]{EKmon}) by the homomorphism $\ZZ^2\rightarrow G$ sending $(1,0)$ to $g_1$ and $(0,1)$ to $g_2$. 
Denote by $\GIII_{2}(G,\gamma,h)$ the grading $\Gamma_\cC(G,\gamma)\otimes\Gamma_\LL$  on $(\cC,\bullet,n)\otimes(\LL,\rho)$. 
The restrictions on $h$ and $\gamma$ ensure that the rank of $\GIII_{2}(G,\gamma,h)$ is $2$.

\item[$\boxed{r=4}$] Pick again an order $3$ element $h\in G$ and another element $g\in G\setminus H$, $H=\langle h\rangle$. 
Set $\gamma=(e,g,g^{-1})$ and consider the grading $\Gamma_\cC(G,\gamma)$ on $\cC$ as in the previous case. 
Denote by $\GIII_{4}(G,g,h)$ the grading $\Gamma_\cC(G,\gamma)\otimes\Gamma_\LL$ on $(\cC,\bullet,n)\otimes(\LL,\rho)$. Its rank is easily checked to be $4$.

\item[$\boxed{r=8}$] Consider the trivial gradings $\Gamma_\cC^\mathrm{triv}$ and $\Gamma_\cO^\mathrm{triv}$ on the para-Cayley algebra and the Okubo algebra,
respectively. Pick an order $3$ element $h\in G$ and denote by $\GIII_{8}(G,h,\mathrm{p})$ the grading 
$\Gamma_\cC^\mathrm{triv}\otimes\Gamma_\LL$ on $(\cC,\bullet,n)\otimes(\LL,\rho)$ and 
by $\GIII_{8}(G,h,\mathrm{o})$ the grading $\Gamma_\cO^\mathrm{triv}\otimes\Gamma_\LL$ on $(\cO,\star,n)\otimes(\LL,\rho)$.
\end{itemize}


\subsection{Classification up to isomorphism}
The next result classifies Type III gradings on the simple Lie algebra of type $D_4$ up to isomorphism by classifying 
Type III gradings on the cyclic composition algebra of rank $8$ up to similarity.
 
\begin{theorem}\label{th:class_isomIII}
Let $\Gamma$ be a Type III grading by an abelian group $G$ on the cyclic composition algebra $(V,\LL,\rho,*,Q)$ of rank $8$ 
over an algebraically closed field $\FF$, $\chr\FF\ne 2,3$. Then $\Gamma$ is similar to one of the gradings 
$\GIII_{0}(G,K,h,\delta)$, $\GIII_{1}(G,K,h)$, $\GIII_{2}(G,\gamma,h)$, $\GIII_{4}(G,g,h)$, or $\GIII_{8}(G,h,\mathrm{t})$, 
where $\delta$ is $+$ or $-$ and $\mathrm{t}$ is $\mathrm{p}$ or $\mathrm{o}$.

Moreover, the gradings with different ranks on the list above are not similar, and for gradings of the same rank we have:
\begin{itemize}
\item If $\GIII_{0}(G,K,h,\delta)$ is similar to $\GIII_{0}(G,K',h',\delta')$, then $K\langle h\rangle=K'\langle h'\rangle$ 
and also $\langle h\rangle=\langle h'\rangle$. Assuming the subgroups $H=\langle h\rangle$ and $KH$ are fixed, there are exactly two similarity classes: 
the gradings $\GIII_{0}(G,K,h,\delta)$ and $\GIII_{0}(G,K',h',\delta')$ are similar if and only if either $\delta'=\delta$ and $h'=h$ 
or $\delta'=-\delta$ and $h'=h^{-1}$.

\item $\GIII_{1}(G,K,h)$ is similar to $\GIII_{1}(G,K',h')$ if and only if $K'=K$ and $\langle h'\rangle=\langle h\rangle$.

\item $\GIII_{2}(G,\gamma,h)$ is similar to $\GIII_{2}(G,\gamma',h')$ if and only if $\langle h'\rangle =\langle h\rangle$ and there exists 
a permutation $\pi\in S_3$ and $1\leq j\leq 3$ such that either $g'_i=g_{\pi(i)}h^j$, for all $i=1,2,3$, or $g'_i=g_{\pi(i)}^{-1}h^j$, for all $i=1,2,3$.

\item $\GIII_{4}(G,g,h)$ is similar to $\GIII_{4}(G,g',h')$ if and only if $\langle h'\rangle=\langle h\rangle$ and $g'$ equals either $g$ or $g^{-1}$.

\item $\GIII_{8}(G,h,\mathrm{t})$ is similar to $\GIII_{8}(G,h',\mathrm{t}')$ if and only if $\langle h'\rangle=\langle h\rangle$ and $\mathrm{t}'=\mathrm{t}$.
\end{itemize}
\end{theorem}
\begin{proof}
We start with the most difficult case, which is the case of rank $0$.
The fact that any grading of Type III and rank $0$ is isomorphic to a grading of the form $\GIII_{0}(G,K,h,\pm)$ follows from 
Theorem \ref{th:gradings_cyclic} and its proof. Now, if $\GIII_{0}(G,K,h,\delta)$ is similar to $\GIII_{0}(G,K',h',\delta')$ then their supports and distinguished subgroups coincide, 
and hence $K\langle h\rangle=K'\langle h'\rangle$ and $\langle h\rangle=\langle h'\rangle$. 
In particular, $K$ is a subgroup in the support of $\GIII_{0}(G,K',h',\delta')$. Hence, as in the proof of Theorem \ref{th:gradings_cyclic}, 
we may consider the $\FF$-subalgebra $\bigoplus_{k\in K} V_k$ in $V=\cO\otimes \LL$, endowed with the grading
$\GIII_{0}(G,K',h',\delta')=\Gamma_\cO(G,K,\delta')\otimes \Gamma_\LL'$, where $\Gamma_\LL'$ is the grading on $\LL$ with $\deg \xi=h'$. 
This is a symmetric composition algebra, graded by a group isomorphic to $\ZZ_3^2$ with one-dimensional homogeneous components, so it is the Okubo algebra. 
This shows that $\GIII_{0}(G,K',h',\delta')$ is isomorphic to the grading $\Gamma_\cO^{\delta''}\otimes\Gamma_\LL'=\GIII_{0}(G,K,h',\delta'')$, 
where $h'$ equals $h$ or $h^{-1}$ and $\delta''\in\{+,-\}$. 
Therefore, once we fix the subgroups $\langle h\rangle$ and $K\langle h\rangle$, we get at most four similarity classes.

Now consider the case $K'=K$, $h'=h^{-1}$, $\delta=+$ and $\delta'=-$. 
Fix two generators $k_1$ and $k_2$ of $K$ (recall that $K$ is isomorphic to $\ZZ_3^2$) and pick elements 
$x,y\in\cO$ with $\deg x=k_1$, $\deg y=k_2$ (with respect to $\Gamma_\cO^+$) and $n(x,x\star x)=1=n(y,y\star y)$. 
By \cite[Lemma 4.47]{EKmon}, we have either $x\star y=0$ or $y\star x=0$ but not both. Interchanging $k_1$ and $k_2$ if necessary, we will assume that $x\star y=0$. 
The grading $\Gamma_\cO^-$ is given by $\deg x=k_2$ and $\deg y=k_1$. Because of \cite[Lemma 4.47]{EKmon}, there is a unique involution $\sigma$ on $\cO$ 
such that $\sigma(x)=y$. Also consider the automorphism $\tau\colon (x_1,x_2,x_3)\mapsto (x_1,x_3,x_2)$ of $\LL$, which takes $\xi$ to $\xi^2$. 
Then the map $\sigma\otimes \tau$ is a graded isomorphism from the opposite of $(\cO,\star,n)\otimes(\LL,\rho)$, 
endowed with the grading $\Gamma_\cO^+\otimes\Gamma_\LL$, onto $(\cO,\star,n)\otimes(\LL,\rho)$, endowed with the grading $\Gamma_\cO^-\otimes\Gamma_\LL'$. 
This shows that we get at most two similarity classes, with representatives $\Gamma_\cO^+\otimes\Gamma_\LL$ and $\Gamma_\cO^+\otimes\Gamma_\LL'$. 

Finally, there is no graded isomorphism or anti-isomorphism from $\Gamma_\cO^+\otimes\Gamma_\LL$ onto $\Gamma_\cO^+\otimes\Gamma_\LL'$. 
Indeed, any such (anti-)isomorphism $(\varphi_1,\varphi_0)$ takes the $\FF$-subalgebra $\cO\otimes 1=\bigoplus_{k\in K}V_k$ of $V=\cO\otimes\LL$ onto itself. 
Since $\cO_{k_1}=\FF x$ and $\cO_{k_2}=\FF y$, we obtain $\varphi_1(x\otimes 1)*\varphi_1(y\otimes 1)=0=\varphi_1\bigl((x\otimes 1)*(y\otimes 1)\bigr)$, 
whereas $0\ne \varphi_1\bigl((y\otimes 1)*(x\otimes 1)\bigr)$. It follows that $(\varphi_1,\varphi_0)$ 
is necessarily an isomorphism, so $\varphi_0\rho=\rho\varphi_0$ and hence $\varphi_0$ must respect the grading on $\LL$, which is a contradiction 
because it takes $\Gamma_\LL$ to $\Gamma_\LL'$. The proof of the rank $0$ case is complete.

\smallskip

If the rank is $8$, then $V_e$ is a symmetric composition algebra of dimension $8$ and hence isomorphic to $(\cC,\bullet,n)$ or $(\cO,\star,n)$. 
This isomorphism class is an invariant of the grading, and the result follows.

\smallskip

If the rank is $1$, then $V_e=\FF \varepsilon$, for a unique idempotent $\varepsilon$. As in the proof of Theorem \ref{th:gradings_cyclic}, 
we may identify $V$ with $\cC\otimes\LL$, where $\cC=\{X\in V\,|\, X*\varepsilon=\bar X\}$. This is the para-Cayley algebra with para-unit $\varepsilon$, and 
it is an invariant of the grading. Besides, $\cC$ is $G$-graded with neutral homogeneous component of dimension $1$. 
The only possibility is that the support of the grading on $\cC$ is isomorphic to $\ZZ_2^3$, and this support is also an invariant of the grading. 
Thus $\Gamma$ is isomorphic to $\GIII_{1}(G,K,h)$. The standard involution $\sigma\colon x\mapsto \bar x$ on $\cC$ preserves the grading on $\cC$, 
and the map $\sigma\otimes\tau$ ($\tau$ as above interchanging $\xi$ and $\xi^2$) gives a graded anti-isomorphism from $(\cC,\bullet,n)\otimes(\LL,\rho)$, 
endowed with the grading $\GIII_{1}(G,K,h)$, onto itself, but endowed with the grading $\GIII_{1}(G,K,h^2)$.

\smallskip

If the rank is $4$, then $V_e$ is a symmetric composition algebra of dimension $4$, so it is a para-quaternion algebra and hence contains a unique 
para-unit $\varepsilon$ (\cite[Proposition 4.43 and Theorem 4.44]{EKmon}), which is thus an invariant of the grading.
The graded para-Cayley algebra $\cC=\{X\in V \,|\,X*\varepsilon=\bar X\}$ is then also an invariant of the grading. Since any grading on the Cayley algebra 
with neutral homogeneous component of dimension $\geq 2$ is induced from the Cartan grading (\cite[Corollary 4.13]{EKmon}), the grading on $\cC$ is of the form $\Gamma_\cC(G,\gamma)$ with $\gamma=(e,g,g^{-1})$ for some $g\in G$. Moreover, the condition $g\not\in \langle h\rangle$ follows from the fact $\dim V_e=4$. The pair $\{g,g^{-1}\}$ is an invariant of the grading, and the gradings $\GIII_{4}(G,g,h)$ and $\GIII_{4}(G,g,h^2)$ are proved to be similar with the same argument as for rank $1$.

\smallskip

Finally, if the rank is $2$, then $V_e$ is the para-quadratic composition algebra $(\cK,\bullet,n)$, and this contains exactly three para-units. 
Actually, $\cK$ is isomorphic to $\FF\times \FF$  but with the para-Hurwitz product: $(x_1,x_2)\bullet (y_1,y_2)=(x_2y_2,x_1y_1)$. 
The para-units are $(1,1)$, $(\omega,\omega^2)$ and $(\omega^2,\omega)$.
If $\varepsilon$ is a para-unit of $V_e$, we may take, as in the proof of Theorem \ref{th:gradings_cyclic}, the para-Cayley algebra 
$\cC=\{ X\in V\,|\, X*\varepsilon=\bar X\}$ with para-unit $\varepsilon$, and then identify $V$ with $\cC\otimes\LL$. 
This para-Cayley subalgebra is graded with $\dim \cC_e=2$, so the grading is induced from the Cartan grading, and hence is of the form 
$\Gamma_\cC(G,\gamma)$ with $\gamma=(g_1,g_2,g_3)$ satisfying $g_1g_2g_3=e$. By \cite[Theorem 4.21]{EKmon}, two such gradings 
$\Gamma_\cC(G,\gamma)$ and $\Gamma_\cC(G,\gamma')$ are isomorphic if and only if there is a permutation $\pi\in S_3$ such that 
either $g_i'=g_{\pi(i)}$ for all $i=1,2,3$, or $g'_i=g_{\pi(i)}^{-1}$ for all $i=1,2,3$.

It follows that our grading of rank $2$ is similar to $\GIII_{2}(G,\gamma,h)$, and, to ensure that $\dim V_e=2$, we must have $g_i\in G\setminus\langle h\rangle$ 
for all $i=1,2,3$. Now consider $\cC$ as the Cayley algebra, with product denoted by juxtaposition. Let $e_1$ and $e_2$ be the two nonzero orthogonal idempotents 
of the quadratic algebra $\cC_e$, and consider the Peirce subspaces $\cU=\{x\in\cC\,|\, e_1x=x=xe_2\}$, $\cV=\{x\in\cC\,|\, e_2x=x=xe_1\}$, 
so that $\cC=\cC_e\oplus\cU\oplus\cV$, and $\cU=\bigoplus_{i=1}^3\cU_{g_i}$, $\cV=\bigoplus_{i=1}^3\cV_{g_i^{-1}}$. 
If we take, instead of $\varepsilon$, the para-unit $\varepsilon'=\omega e_1+\omega^2 e_2$ of $\cC_e$, then the corresponding para-Cayley subalgebra is
\[
\begin{split}
\cC'&=\{X\in V\,|\, X*\bigl(\varepsilon'\otimes 1\bigr)=
b_Q\bigl(\varepsilon'\otimes 1,X\bigr)\bigl(\varepsilon'\otimes 1\bigr) -X\}\\
 &=(\cC_e\otimes 1) \oplus (\cU\otimes\xi) \oplus (\cV\otimes \xi^2).
\end{split}
\]
This shows that $\GIII_{2}(G,(g_1,g_2,g_3),h)$ is isomorphic to $\GIII_{2}(G,(hg_1,hg_2,hg_3),h)$, and also to $\GIII_{2}(G,(h^2g_1h^2g_2,h^2g_3),h)$.
As in the cases of rank $1$ and $4$, we see that $\GIII_{2}(G,\gamma,h)$ and $\GIII_{2}(G,\gamma,h^2)$ are similar, and the result follows.
\end{proof}

\section{Appendix: graded modules}\label{s:graded_mod}

In what follows, we assume the ground field $\FF$ to be {\em algebraically closed of characteristic $0$}. In \cite{EK14}, graded modules for the classical simple Lie algebras were studied. However, for the simple Lie algebra $\cL$ of type $D_4$, the computation of graded Brauer invariants of irreducible $\cL$-modules was restricted to the case when the $G$-grading on $\cL$ is a {\em matrix grading}, i.e., when $\wh{G}$ fixes the isomorphism class of the natural module. This covers the cases of Type I and II gradings. (As mentioned in the Introduction, Type II reduces to a matrix grading.) Our goal in this appendix is to complete the results of \cite{EK14} by considering Type III gradings on $\cL$.
We already showed in Subsection \ref{sse:lifting1} that, in this case, the graded Brauer invariants of the natural and half-spin modules are equal to the identity element of the $(G/H)$-graded Brauer group of $\FF$. But this is not sufficient to obtain the graded Brauer invariant of every irreducible module. In order to complete the calculation, we will need a few general remarks of independent interest.

\subsection{Background on algebraic groups}

Let $\cL$ be a finite-dimensional semisimple Lie algebra over $\FF$. The corresponding adjoint algebraic group $\bar\cG$ is the group of inner automorphisms $\inaut(\cL)$. Denote by $\tilde\cG$ the associated simply connected group. Once we fix a Borel subgroup $\wt{B}$ in $\tilde{\cG}$ and a maximal torus $\wt{T}$ in $\wt{B}$, we obtain the root system $\Phi$ of $\cL$ and a system of simple roots $\Pi$. The group of characters $\mathfrak{X}(\wt{T})$ is the group of integral weights $\Lambda$. Denote by $\Lambda^+$ the subset of dominant weights and by $\Lambda^{\mathrm{r}}$ the root lattice: $\Lambda^{\mathrm{r}}=\ZZ \Phi$.

For any (connected) semisimple algebraic group $\cG$ with the same root system, there are isogenies $\tilde\cG\xrightarrow{\tilde\pi}\cG\xrightarrow{\bar\pi}\bar\cG$. Moreover (see e.g. \cite[Theorem 8.17]{MT}), with $T=\tilde\pi(\wt{T})$, we have the root space decomposition:
\[
\cL=\Lie(\cG)=\Lie(T)\oplus\Bigl(\bigoplus_{\alpha\in\Phi}\cL_\alpha\Bigr),
\]
and for each root $\alpha\in\Phi$ there is a closed imbedding of algebraic groups $u_\alpha\colon\Gs_a\rightarrow \cG$ such that $tu_\alpha(c)t^{-1}=u_\alpha\bigl(\alpha(t)c\bigr)$ for all $t\in T$ and $c\in\FF$. As usual, we denote by $\Gs_a$ the additive group of $\FF$ and by $\Gs_m$ the one-dimensional torus $\FF^\times=\GL_1(\FF)$. 
Once a Chevalley basis $\{h_i,x_\alpha\mid i=1,\ldots,\rank(\cL),\ \alpha\in\Phi\}$ is fixed, if $V$ is a faithful rational module for $\cG$ and we identify $\cG$ with a subgroup of $\GL(V)$, and hence $\cL$ with a subalgebra of $\frgl(V)$, then we may take $u_\alpha(c)=\exp(cx_\alpha)$ (see \cite{Steinberg}).

The group of characters $\mathfrak{X}(T)$ is a lattice with $\Lambda^{\mathrm{r}}=\mathfrak{X}(\wb{T})\leq \mathfrak{X}(T)\leq \Lambda=\mathfrak{X}(\wt{T})$, where $\wb{T}=\bar{\pi}(T)$. Given any symmetry of the Dynkin diagram $\tau\in\Aut(\mathrm{Dyn})$ such that $\tau$ preserves $\mathfrak{X}(T)$, there is a unique automorphism $\sigma_\tau\in\Aut(\cG)$ such that $\sigma_\tau\bigl(u_\alpha(c)\bigr)=u_{\tau(\alpha)}(c)$ for all $\alpha\in\Pi$ and $c\in\FF$ (see e.g. \cite[\S 23.7]{Chevalley} or \cite[p.~156]{Steinberg}). For the adjoint $\bar\cG$ or simply connected $\tilde \cG$, this allows the construction of the semidirect products $\bar\cG\rtimes \Aut(\mathrm{Dyn})$ and 
$\tilde\cG\rtimes \Aut(\mathrm{Dyn})$. The first one is isomorphic to the automorphism group $\Aut(\cL)$, where to any $\tau\in\Aut(\mathrm{Dyn})$ as above we associate the unique automorphism of $\cL$, also denoted by $\sigma_\tau$, such that $\sigma_\tau(x_\alpha)=x_{\tau(\alpha)}$ for any $\alpha\in\Pi$ (see \cite[Chapter~IX]{Jacobson}).

Given a dominant weight $\lambda\in\Lambda^+$, consider its stabilizer $S_\lambda$ in $\Aut(\mathrm{Dyn})$. For $\tau\in S_\lambda$, let $\sigma=\sigma_\tau$ be the associated automorphism of $\cL$. Then $\sigma$ extends to an automorphism of the universal enveloping algebra $U(\cL)$ that preserves the left ideal $J(\lambda)$ in \cite[\S 21.4]{Humphreys}, and hence induces an element, also denoted by $\sigma$, in $\GL(V_\lambda)$, where $V_\lambda=U(\cL)/J(\lambda)$ is the irreducible module with highest weight $\lambda$. 
By definition of $\sigma$, we have $\sigma(xv)=\sigma(x)\sigma(v)$ for all $x\in\cL$ and $v\in V_\lambda$.

The corresponding representation $\rho\colon\cL\rightarrow \frgl(V_\lambda)$ ``integrates'' to a representation $\tilde\rho\colon\tilde\cG\rightarrow \GL(V_\lambda)$ that extends to
\begin{equation}\label{eq:rhotilde}
\tilde\rho\colon\tilde\cG\rtimes S_\lambda\rightarrow \GL(V_\lambda).
\end{equation}

\subsection{Graded Brauer invariants of irreducible modules}

Let $G$ be an abelian group and let $\Gamma:\cL=\bigoplus_{g\in G}\cL_g$ be a $G$-grading on $\cL$. Recall that the grading is determined by a morphism of affine group schemes, $\eta_\Gamma\colon G^D\rightarrow \AAut(\cL)$. Because of our assumptions on the ground field $\FF$, it is sufficient to consider only the $\FF$-points, i.e., a morphism of algebraic groups $\wh{G}\to\Aut(\cL)$.
Strictly speaking, $\wh{G}$ is an algebraic group only if $G$ is finitely generated, while  in general it is a pro-algebraic group. However, since we are dealing with gradings on finite-dimensional objects (algebras and modules), we may replace $G$ by a finitely generated subgroup in all arguments that deal with finitely many objects.

Following \cite{EK14}, we denote the image of a character $\chi\in\widehat{G}$ by $\alpha_\chi$, i.e., $\alpha_\chi(x)\bydef \chi(g)x$ for all $g\in G$ and $x\in\cL_g$. Thus any character $\chi\in\widehat{G}$ induces the automorphism $\alpha_\chi\in\Aut(\cL)$ and an associated diagram automorphism $\tau_\chi\in\Aut(\mathrm{Dyn})=\Aut(\cL)/\inaut(\cL)$. As in \cite{EK14}, given a dominant weight $\lambda\in\Lambda^+$, consider the \emph{inertia group}:
\[
K_\lambda\bydef \{ \chi\in\widehat{G} \mid \tau_\chi(\lambda)=\lambda\}.
\]
This is the inverse image of the subgroup $\bar{\cG}\rtimes S_\lambda\subset\bar\cG\rtimes\Aut(\mathrm{Dyn})\simeq\Aut(\cL)$ under $\eta_\Gamma$. We also write $H_\lambda\bydef K_\lambda^\perp\subset G$, where $\perp$ denotes the ``orthogonal complement'' in $\wh{G}$ of a subgroup of $G$, or vice versa.

Let $\pi\colon\tilde\cG\rtimes\Aut(\mathrm{Dyn})\rightarrow \bar\cG\rtimes\Aut(\mathrm{Dyn})$ be the natural projection, and consider preimages $\tilde\alpha_\chi$ in $\tilde\cG\rtimes\Aut(\mathrm{Dyn})$ for all $\chi\in\widehat{G}$. Recall that the kernel of $\pi$ equals the center of the simply connected group, $Z(\tilde\cG)$. Since $\widehat{G}$ is abelian, the commutator $[\alpha_{\chi_1},\alpha_{\chi_2}]=\alpha_{\chi_1}\alpha_{\chi_2}\alpha_{\chi_1}^{-1}\alpha_{\chi_2}^{-1}$ is trivial, and hence we obtain $[\tilde\alpha_{\chi_1},\tilde\alpha_{\chi_2}]\in Z(\tilde\cG)$ for all $\chi_1,\chi_2\in\widehat{G}$.
The action of the algebraic group $\Aut(\cL)=\bar\cG\rtimes\Aut(\mathrm{Dyn})$, or of $\tilde\cG\rtimes\Aut(\mathrm{Dyn})$, on $\cL=\Der(\cL)$ is the adjoint action. In particular, we have
\begin{equation}\label{eq:rhotildealpha}
\rho\bigl(\alpha_\chi(x)\bigr)=\tilde\rho(\tilde\alpha_\chi)\rho(x)\tilde\rho(\tilde\alpha_\chi)^{-1},
\end{equation}
for any $\chi\in K_\lambda$ and $x\in\cL$. On the other hand, the elements in the center $Z(\tilde\cG)$ act by scalars on $V_\lambda$, so there is an associated morphism 
\[
\Psi_\lambda\colon Z(\tilde\cG)\rightarrow \Gs_m,
\] 
defined by $\tilde\rho(z)=\Psi_\lambda(z)\id$ for all $z\in Z(\tilde\cG)$. In other words, $\Psi_\lambda$ is the restriction of $\lambda\in\mathfrak{X}(\wt{T})$ to $Z(\tilde\cG)$. 

Recall from Section \ref{s:preliminaries} that the elements of the $G$-graded Brauer group $B_G(\FF)$ can be interpreted as alternating bicharacters $\hat{\beta}\colon\wh{G}\times \wh{G}\to\Gs_m$ (which factor through the homomorphism $\wh{G}\to\wh{G_0}$ where $G_0$ is the torsion subgroup of $G$). 
There is a unique $(G/H_\lambda)$-grading on the associative algebra $\End(V_\lambda)$ such that $\rho\colon\cL\to\frgl(V_\lambda)$ is a homomorphism of $(G/H_\lambda)$-graded algebras. The \emph{Brauer invariant} $\Br(\lambda)$ (see \cite[Definition 4]{EK14}) is the class $[\End(V_\lambda)]$ in $B_{G/H_\lambda}(\FF)$. Together with the subgroup $H_\lambda$, it measures how far the irreducible module $V_\lambda$ is from admitting a $G$-grading compatible with the $G$-grading on $\cL$. To be precise, $V_\lambda$ admits such a grading if and only if $H_\lambda=1$ and $\Br(\lambda)$ is trivial. Moreover, knowing $H_\lambda$ and $\Br(\lambda)$ for all $\lambda\in\Lambda^+$ allows us to classify the simple objects in the category of finite-dimensional graded $\cL$-modules (see \cite[Theorem 8]{EK14}). 
Using the above interpretation of graded Brauer groups, $\Br(\lambda)$ is identified with the commutation factor of the $(G/H_\lambda)$-graded matrix algebra $\End(V_\lambda)$, which is the alternating bicharacter $\hat{\beta}_\lambda\colon K_\lambda\times K_\lambda\rightarrow \Gs_m$ defined as follows. For every $\chi\in K_\lambda$, there exists an invertible element $u_\chi\in\End(V_\lambda)$, unique up to a scalar multiple, such that the action of $\chi$ on $\End(V_\lambda)$ (associated to the $(G/H_\lambda)$-grading) is the inner automorphism 
$a\mapsto u_\chi a u^{-1}_\chi$, and then $\hat{\beta}_\lambda$ is defined by the equation $u_{\chi_1}u_{\chi_2}=\hat{\beta}_\lambda(\chi_1,\chi_2)u_{\chi_2}u_{\chi_1}$ for all $\chi_1,\chi_2\in K_\lambda$. 
In view of Equation \eqref{eq:rhotildealpha}, we can take $u_\chi=\tilde{\rho}(\tilde{\alpha}_\chi)$, so we obtain a new interpretation of $\hat{\beta}_\lambda$, namely,
\begin{equation}\label{eq:new}
\hat{\beta}_\lambda(\chi_1,\chi_2)=\Psi_\lambda\bigl([\tilde\alpha_{\chi_1},\tilde\alpha_{\chi_2}]\bigr).
\end{equation}
This new point of view on Brauer invariants has the following consequences:

\begin{proposition}\label{pr:Br_trivial}
Let $\cL$ be a semisimple Lie algebra, with Dynkin diagram $\mathrm{Dyn}$, weight lattice $\Lambda$ and root lattice $\Lambda^{\mathrm{r}}$, endowed with a grading by an abelian group $G$. Let $\lambda\in\Lambda$ be a dominant weight.
\begin{enumerate}
\item If $\lambda\in\Lambda^{\mathrm{r}}$, then $\Br(\lambda)$ is trivial.
\item If $\inaut(\cL)$ is simply connected (i.e., $\Lambda=\Lambda^{\mathrm{r}}$) and $\Aut(\mathrm{Dyn})$ is trivial, then any $\cL$-module admits a compatible $G$-grading.
\end{enumerate}
\end{proposition}
\begin{proof}
We use the notation introduced in the previous subsection: $\cG$ is a connected algebraic group with $\Lie(\cG)=\cL$, etc. Recall that the isomorphism classes of irreducible representations of $\cG$ correspond bijectively to the weights in $\mathfrak{X}(T)\cap\Lambda^+$. In particular,
if $\lambda\in\Lambda^{\mathrm{r}}$, then the representation $\rho\colon\cL\rightarrow\frgl(V_\lambda)$ ``integrates'' to a representation $\bar\rho\colon\bar\cG\rtimes S_\lambda\rightarrow \GL(V_\lambda)$, so that $\tilde\rho=\bar\rho\circ \pi$, with $\tilde\rho$ as in \eqref{eq:rhotilde}. 
Hence, $\tilde\rho\bigl(Z(\tilde\cG)\bigr)$ is trivial, so for any $\chi_1,\chi_2\in K_\lambda$, we have $\tilde\rho\bigl([\tilde\alpha_{\chi_1},\tilde\alpha_{\chi_2}]\bigr)=\id$ and thus $\hat\beta_\lambda(\chi_1,\chi_2)=\Psi_\lambda\bigl([\tilde\alpha_{\chi_1},\tilde\alpha_{\chi_2}]\bigr)=1$.

For the second part, if $\Aut(\mathrm{Dyn})$ is trivial, then $K_\lambda=\widehat{G}$ and $H_\lambda=1$ for any $\lambda\in\Lambda^+$. Moreover, $Z(\tilde\cG)=1$ here, and hence $\Br(\lambda)=1$.
\end{proof}

\begin{corollary}
Let $G$ be an abelian group and let $\cL$ be a simple Lie algebra of type $G_2$, $F_4$ or $E_8$, endowed with a $G$-grading. 
Then any finite-dimensional module for $\cL$ admits a compatible $G$-grading.
\end{corollary}

In \cite{EK14}, this was remarked only for type $G_2$ (with a different argument).

\subsection{Brauer invariants for a Type III grading on the simple Lie algebra of type $D_4$}

Theorems 46 and 48 in \cite{EK14}, which compute Brauer invariants for Type I and II gradings on simple Lie algebras of series $D$, can now be completed with the next result, where $\omega_1,\omega_2,\omega_3,\omega_4$ denote the fundamental dominant weights of the simple Lie algebra of type $D_4$, with $\omega_1$, $\omega_3$ and $\omega_4$ corresponding to the natural and half-spin representations (the outer nodes of the Dynkin diagram), and $\omega_2$ to the adjoint representation (the central node of the diagram).

\begin{theorem}
Let $\cL$ be the simple Lie algebra of type $D_4$ over an algebraically closed field $\FF$ of characteristic $0$. Suppose $\cL$ is graded by an abelian group $G$ and that the grading is of Type III. Let $K=\langle h\rangle^\perp$, where $h\in G$ is the distinguished element (see Section \ref{s:lifting}). Then, for a dominant integral weight $\lambda=\sum_{i=1}^4 m_i\omega_i$, we have the following possibilities:
\begin{enumerate}
\item If $m_1=m_3=m_4$, then $H_\lambda=1$, $K_\lambda=\widehat{G}$ and $\Br(\lambda)=1$.
\item Otherwise $H_\lambda=\langle h\rangle$, $K_\lambda=K$ and $\Br(\lambda)=1$. 
\end{enumerate}
\end{theorem}
\begin{proof}
If $m_1=m_3=m_4$, then the diagram automorphisms of order $3$ preserve $\lambda$, so $K_\lambda=\widehat{G}$ and $H_\lambda=1$. Moreover, both $\omega_2$ (the highest root) and $\omega_1+\omega_3+\omega_4$ are in the root lattice $\Lambda^{\mathrm{r}}$, so $\Br(\lambda)=1$ by Proposition \ref{pr:Br_trivial}.

Otherwise we get $K_\lambda=K$, $H_\lambda=\langle h\rangle$, and the associated grading by $\wb{G}\bydef G/\langle h\rangle$ is of Type I. As shown in Subsection \ref{sse:lifting1}, we have $\Br(\omega_1)=\Br(\omega_3)=\Br(\omega_4)=1$ in the $\wb{G}$-graded Brauer group. But, as we just observed, $\omega_2\in \Lambda^{\mathrm{r}}$, so $\Br(\omega_2)=1$, too. The result follows now from \cite[Proposition 10]{EK14}.
\end{proof}

\begin{corollary}
The simple $\cL$-module $V_\lambda$ admits a $G$-grading making it a graded $\cL$-module if and only if $m_1=m_3=m_4$.
\end{corollary}

\subsection{Analogy with Tits algebras}

Tits algebras (see e.g. \cite{Tits} or \cite[\S 27]{KMRT}) were introduced to study representations of semisimple algebraic groups over an arbitrary field. 
In the present work, the ground field is assumed algebraically closed and of characteristic $0$, but we are interested in graded representations.

As before, let $G$ be a finitely generated abelian group and let $\cL$ be a semisimple Lie algebra endowed with a $G$-grading. Consider the associated morphism $\eta\colon\wh{G}\rightarrow \Aut(\cL)=\bar\cG\rtimes\Aut(\mathrm{Dyn})$, where $\bar \cG=\inaut(\cL)$. Any $\chi\in\wh{G}$ induces an automorphism $\alpha_\chi\in\Aut(\cL)$ and hence a diagram automorphism $\tau_\chi\in\Aut(\mathrm{Dyn})=\Aut(\cL)/\inaut(\cL)$. Thus $\wh{G}$ acts on $\Lambda$, $\Lambda^+$ and $\Lambda^{\mathrm{r}}$.

Let $H$ be the finite subgroup of $G$ such that $H^\perp$ is the inverse image of $\inaut(\cL)$ under $\eta$. Fix any subgroup $M$ of $G$ contained in $H$ and denote   $K=M^\perp\subset\widehat{G}$ and $\wb{G}=G/M$. For any $\lambda\in\Lambda^+$ with $K\subset K_\lambda$, or, equivalently, $H_\lambda\subset M$, or $\lambda\in(\Lambda^+)^K$ (the set of dominant weights fixed under the action of $K$), we may consider the Brauer invariant of $\lambda$ in the $\wb{G}$-graded Brauer group $B_{\wb{G}}(\FF)$. Denote it by $\Br_{\wb{G}}(\lambda)$. In this way we get a map
\[
\begin{split}
(\Lambda^+)^K&\longrightarrow B_{\wb{G}}(\FF),\\
\lambda\quad &\mapsto\quad \Br_{\wb{G}}(\lambda),
\end{split}
\]
which is multiplicative by \cite[Proposition 10]{EK14} or directly from Equation \eqref{eq:new}.
Moreover, if $\lambda\in(\Lambda^{\mathrm{r}})^K$, then $\Br_{\wb{G}}(\lambda)=1$ by Proposition \ref{pr:Br_trivial}. 

As in \cite[Corollary 3.5]{Tits} or \cite[Theorem 27.7]{KMRT}, this map ``extends'' to a group homomorphism
\[
\beta_{\wb{G}}\colon(\Lambda/\Lambda^{\mathrm{r}})^K\rightarrow B_{\wb{G}}(\FF),
\]
where for any $\lambda+\Lambda^{\mathrm{r}}\in(\Lambda/\Lambda^{\mathrm{r}})^K$ we consider the unique minimal weight $\hat\lambda\in \Lambda^+$ in the same class modulo $\Lambda^{\mathrm{r}}$ and define $\beta_{\wb{G}}(\lambda+\Lambda^{\mathrm{r}})\bydef\Br_{\wb{G}}(\hat\lambda)$.

\begin{remark} 
The Brauer invariant $\Br(\lambda)$ of $\lambda\in \Lambda^+$ is precisely $\beta_{G/H_\lambda}(\lambda+\Lambda^{\mathrm{r}})$.
\end{remark}

In the setting of \cite{Tits}, a simply connected semisimple group over an arbitrary field $\FF$ is considered, and the group that acts on $\Lambda$, $\Lambda^+$ and $\Lambda^{\mathrm{r}}$ is the absolute Galois group $\Gamma$ of $\FF$, obtaining a group homomorphism $\beta\colon(\Lambda/\Lambda^{\mathrm{r}})^\Gamma\rightarrow B(\FF)$ (the classical Brauer group). Also, if $\lambda+\Lambda^{\mathrm{r}}\in \Lambda/\Lambda^{\mathrm{r}}$ is not $\Gamma$-invariant, then one has to consider the subgroup $\Gamma_\lambda\bydef\{\gamma\in\Gamma\;|\;\gamma(\lambda+\Lambda^{\mathrm{r}})=\lambda+\Lambda^{\mathrm{r}}\}$, the field $\FF_\lambda\bydef(\FF_{\mathrm{sep}})^{\Gamma_\lambda}$ and the homomorphism $(\Lambda/\Lambda^{\mathrm{r}})^{\Gamma_\lambda}\rightarrow B(\FF_\lambda)$. This is analogous to what we did restricting from $\wh{G}$ to $K_\lambda$.

\section*{Acknowledgments}

The second author would like to thank the Instituto Universitario de Matem\'aticas y Aplicaciones and Departamento de Matem\'aticas of the University of Zaragoza for support and hospitality during his visit in June 2014. Also, the authors are grateful to Kirill Zainoulline of the University of Ottawa for pointing out the analogy with Tits algebras.


\end{document}